\documentclass[10pt,a4paper,twoside,reqno]{amsart}

\usepackage{xcolor,soul,lipsum}
\usepackage{hyperref,url}
\usepackage{amssymb,latexsym}
\usepackage{amsmath,amsthm}
\usepackage{enumerate}
\usepackage[all]{xy} \CompileMatrices
\usepackage{amscd}
\hypersetup{colorlinks=false,pdfborderstyle={/S/U/W 0}}
\usepackage{slashed} 
\usepackage{tikz}
\usetikzlibrary{matrix}
\usepackage[normalem]{ulem}
\SelectTips{cm}{12} 
\theoremstyle{plain}
\newtheorem{theorem}{Theorem}[section]
\newtheorem{lemma}[theorem]{Lemma}
\newtheorem{corollary}[theorem]{Corollary}
\newtheorem{proposition}[theorem]{Proposition}

\theoremstyle{definition}
\newtheorem{definition}[theorem]{Definition}
\newtheorem{definition-theorem}[theorem]{Definition-Theorem}
\newtheorem{example}[theorem]{Example}
\theoremstyle{remark}
\newtheorem{remark}[theorem]{Remark}

\numberwithin{equation}{section} 

\usepackage{geometry}
\usepackage{enumitem}

\usepackage[utf8]{inputenc}
\usepackage{color}

\newcommand{\tr}{\operatorname{tr}}
\newcommand{\Id}{\operatorname{Id}}

\newcommand{\End}{\operatorname{End}}

\newcommand{\Ker}{\operatorname{Ker}}

\newcommand{\Ad}{\operatorname{Ad}}

\newcommand{\surj}{\to\kern-1.8ex\to}

\newcommand{\cF}{\mathcal{F}}

\newcommand{\cA}{\mathcal{A}}
\newcommand{\cK}{\mathcal{K}}

\newcommand{\cV}{\mathcal{V}}
\newcommand{\Diff}{\mathrm{Diff}}
\newcommand{\cC}{\mathcal{C}}
\newcommand{\Conf}{\mathrm{Conf}}
\newcommand{\Sol}{\mathrm{Sol}}
\newcommand{\Met}{\mathrm{Met}}
\newcommand{\s}{\mathrm{S}}

\newcommand{\G}{\mathrm{G}}
\newcommand{\cH}{\mathcal{H}}
\newcommand{\cE}{\mathcal{E}}
\newcommand{\cO}{\mathcal{O}}

\newcommand{\arxiv}[1]{{\tt
		\href{http://www.arXiv.org/abs/#1}{arXiv:#1}}}
\newcommand{\norm}[1]{\left\lVert#1\right\rVert}
\newcommand{\cS}{\mathcal{S}}
\newcommand{\eqdef}{\stackrel{{\rm def.}}{=}}
\newcommand{\cI}{\mathcal{I}}
\newcommand{\dd}{\mathrm{d}}

\newcommand{\mF}{\mathrm{F}}
\newcommand{\mR}{\mathcal{R}}
\newcommand{\frg}{\mathfrak{g}}
\newcommand{\frc}{\mathfrak{c}}
\newcommand{\frv}{\mathfrak{v}}
\newcommand{\frs}{\mathfrak{s}}
\newcommand{\fra}{\mathfrak{a}}

\begin{document}

\title[Heterotic solitons on four-manifolds]{Heterotic solitons on four-manifolds}
\author[Andrei Moroianu]{Andrei Moroianu}
\author[\'Angel Murcia]{\'Angel Murcia}
\author[C. S. Shahbazi]{C. S. Shahbazi}

\address{Universit\'e Paris-Saclay, CNRS,  Laboratoire de math\'ematiques d'Orsay, 91405, Orsay, France}
\email{andrei.moroianu@math.cnrs.fr}

\address{Instituto de F\'isica Te\'orica  UAM/CSIC, Espa\~na}
\email{angel.murcia@csic.es}

\address{Fachbereich Mathematik, Universit\"at Hamburg, Deutschland}
\email{carlos.shahbazi@uni-hamburg.de}

\thanks{C.S.S. would like to thank J. Streets and Y. Ustinovskiy for their useful comments on the notion of generalized Ricci soliton. Part of this work was undertaken during a visit of C.S.S. to the University Paris-Saclay under the Deutsch-Franz\"osische  Procope Mobilit\"at program. C.S.S. would like to thank A. Moroianu and this very welcoming institution for providing a nice and stimulating working environment. The work of \'A.M. was funded by the Spanish FPU Grant No. FPU17/04964, with additional support from the MCIU/AEI/FEDER UE grant PGC2018-095205-B-I00 and the Centro de Excelencia Severo Ochoa Program grant SEV-2016-0597. The work of C.S.S. was supported by the Germany Excellence Strategy \emph{Quantum Universe} - 390833306.}

\maketitle

\begin{abstract}
We investigate four-dimensional Heterotic solitons, defined as a particular class of solutions of the equations of motion of Heterotic supergravity on a four-manifold $M$ or, equivalently, as self-similar points of the renormalization group flow of the NS-NS sector of the Heterotic world-sheet. Heterotic solitons depend on a parameter $\kappa$ and consist of a Riemannian metric $g$, a metric connection with skew torsion $H$ on $TM$ and a closed 1-form $\varphi$ on $M$ satisfying a differential system that generalizes the celebrated Hull-Strominger system. In the limit $\kappa \to 0$, Heterotic solitons reduce to a class of generalized Ricci solitons and can be considered as a higher-order curvature modification of the latter. If the torsion $H$ is equal to the Hodge dual of $\varphi$, Heterotic solitons consist of either flat tori or closed Einstein-Weyl structures on manifolds of type $S^1\times S^3$ as introduced by P. Gauduchon. We prove that the moduli space of such closed Einstein-Weyl structures is isomorphic to the product of $\mathbb{R}$ with a certain finite quotient of the Cartan torus of the isometry group of the typical fiber of a natural fibration $M\to S^1$. We also consider the associated space of essential infinitesimal deformations, which we prove to be obstructed. More generally, we characterize several families of Heterotic solitons as suspensions of certain three-manifolds with prescribed constant principal Ricci curvatures, amongst which we find hyperbolic manifolds, manifolds covered by $\widetilde{\mathrm{Sl}}(2,\mathbb{R})$ and E$(1,1)$ or certain Sasakian three-manifolds. These solutions exhibit a topological dependence in the string slope parameter $\kappa$ and yield, to the best of our knowledge, the first examples of Heterotic compactification backgrounds not locally isomorphic to supersymmetric compactification backgrounds. 
\end{abstract}

\setcounter{tocdepth}{1} 


\section{Introduction}
\label{sec:intro}


The goal of this article is to investigate a system of partial differential equations, which we call the \emph{Heterotic soliton system}, that occurs naturally as part of the equations of motion of the bosonic sector of Heterotic supergravity in four dimensions \cite{BRI,BRII,Ortin}. The latter is defined on a principal bundle $P$ over a four-manifold $M$ and involves a Riemannian metric $g$ on $M$, a pair $(\varphi,\alpha)$ of 1-forms on $M$ and a connection $A$ on $P$ coupled through a system of highly non-linear partial differential equations completely determined by \emph{supersymmetry}. The bosonic sector of Heterotic supergravity generalizes the Einstein-Yang-Mills system and contains, through its Killing spinor equations, the celebrated Hull-Strominger system \cite{Hull,Strominger}. Despite the fact that the four-dimensional supersymmetric solutions of Heterotic supergravity have been fully classified in \cite{GFRST,Strominger}, the classification of all possibly non-supersymmetric solutions of the theory on a compact four-manifold seems to be currently out of reach and in fact, and to the best of our knowledge, no non-locally supersymmetric compactification background of Heterotic supergravity was known prior to this work. On the other hand, in Euclidean dimensions higher than four, the existence, uniqueness and moduli problems of Heterotic supersymmetric solutions remain wide open and have attracted extensive attention in the physics as well as in the mathematics literature, see for instance the reviews \cite{Fei,GFR,Grana:2005jc,Phong,YauIII} and their references and citations for more details. In this regard, Yau's conjecture on the existence of solutions to the Hull-Strominger system on certain polystable holomorphic vector bundles over compact balanced complex manifolds stands as an outstanding open problem in the field \cite{YauV,YauVI,YauIV}. 

Given the complexity of the four-dimensional full-fledged Heterotic bosonic sector, in this work we propose an educated truncation which is obtained by taking the structure group of the \emph{gauge bundle} $P$ to be trivial, that is, $P = M$. With this assumption, the bosonic sector of Heterotic supergravity reduces to a system of partial differential equations for a Riemannian metric and a pair of 1-forms $\varphi$ and $\alpha$ on a four-manifold $M$  which we call the \emph{Heterotic soliton system (see Definition \ref{def:HeteroticSoliton}) and which} is given by:
\begin{eqnarray*}
&\mathrm{Ric}^{g} +  \nabla^{g}\varphi + \frac{1}{2} \alpha\otimes \alpha - \frac{1}{2} \vert \alpha \vert^2_g\, g + \kappa \,\mathfrak{v}(\mR_{\nabla^{\alpha}} \circ\mR_{\nabla^{\alpha}}) = 0\, , \quad \dd\alpha = \varphi\wedge \alpha\\ 
& \delta^g \varphi + |\varphi|^2_g + \kappa \,|\mR_{\nabla^{\alpha}}|^2_{g,\mathfrak{v}}   = |\alpha|^2_g \, , \qquad \delta^g \alpha = \kappa \,(|\mR^{+}_{\nabla^{\alpha}}|^2_{g,\mathfrak{v}}  - |\mR^{-}_{\nabla^{\alpha}}|^2_{g,\mathfrak{v}})\, ,
\end{eqnarray*}

\noindent
where $\nabla^{\alpha}$ denotes the unique metric connection on $(M,g)$ with skew-symmetric torsion $-(\ast_g \alpha)$ and $\mathfrak{v}(- \circ_g -)$ is a bilinear algebraic operation introduced in Section \ref{subsec:preliminaries}. Solutions of this system are by definition \emph{Heterotic solitons} on $M$, see Definition \ref{def:HeteroticSoliton}. Heterotic solutions depend on a non-negative constant $\kappa\geq 0$, which corresponds physically to the \emph{slope parameter} of the Heterotic string to which the theory corresponds. In the limit $\kappa \to 0$, Heterotic solitons reduce to a particular class of generalized Ricci solitons as introduced in \cite{FernandezStreets}. The latter can be understood as stationary points of generalized Ricci flow \cite{Oliynyk:2005ak,Streets1,Streets2}, which originates through the renormalization group flow of the NS-NS string at one loop \cite[Page 111]{Polchinski:1998rq} and in the \emph{Hamilton gauge} \cite{Oliynyk:2005ak}. In the same vein, Heterotic solitons correspond to self-similar solutions of the generalized Ricci flow corrected by higher loops in $\kappa$, which turn out to introduce higher curvature terms in the system of equations. Therefore, Heterotic solitons can be understood as a natural extension of the notion of generalized Ricci solitons in the context of Heterotic string theory. The investigation of flow equations inspired by supergravity and superstring theories is an increasingly active topic of research in the mathematics literature, see \cite{Fei:2018mzf,Fei:2020kkl,Fei:2020nwv,Phong:2018wgn,PhongPicardZhang,StreetsA} and references therein, and the characterization of the renormalization group flow whose self-similar solutions are Heterotic solitons is currently work in progress and has already produced a novel curvature flow \cite{HeteroticRicciFlow}.  

Having introduced the notion of Heterotic soliton, which seems to be new in the literature, our first goal is to construct non-trivial examples and study the associated moduli space of solutions in simple cases. Heterotic solitons $(g,\varphi,\alpha)$ with $\varphi = \alpha$ can be easily proven to be manifolds of type $S^1\times S^3$ as introduced by P. Gauduchon in \cite{Gauduchon} (see Definition \ref{def:Gauduchon}), which in turn leads us to revisit Reference \cite{Pedersen} and reconsider the study of the moduli space of such manifolds. Our first result in this direction is the following.

\begin{theorem}
Let $\Sigma$ be a spherical three-manifold. The moduli space of manifolds of type $S^1\times S^3$ and class $\Sigma$ is in bijection with the direct product of $\mathbb{R}$ with a finite quotient of a maximal torus $T$ in the isometry group of $\Sigma$. In particular, the moduli space of manifolds of type $S^1\times S^3$ has dimension $1 + \mathrm{rk}(\mathrm{Iso}(\Sigma))$, where $\mathrm{rk}(\mathrm{Iso}(\Sigma))$ denotes the rank of $\mathrm{Iso}(\Sigma)$, that is, the dimension of any of its maximal torus subgroups.
\end{theorem}

\noindent
The reader is referred to Theorem \ref{thm:bijection} for more details. The previous theorem characterizes the moduli space of manifolds of type $S^1\times S^3$ \emph{globally}. Since such type of characterization is relatively rare in differential-geometric moduli problems, we perform in addition a local study of the moduli, characterizing its virtual tangent space $T_{[g,\varphi]}\mathfrak{M}^0_{\omega}(M)$ of  infinitesimal deformations that preserve the norm of $\varphi$, chosen to be 1, and the Riemannian volume $\omega$ form of $g$.  This eliminates trivial deformations such as constant rescalings of $\varphi$ and $g$, and is also called the vector space of \emph{essential} deformations, according to the terminology introduced by N. Koiso \cite{Koiso,KoisoII}.
\begin{theorem}
There exists a canonical bijection:
\begin{equation*}
T_{[g,\varphi]}\mathfrak{M}^0_{\omega}(M) \to   \cK(\Sigma)\, ,  
\end{equation*}
\noindent
where the Riemannian three-manifold $\Sigma$ is the typical fiber of the natural fibration structure determined by $(g,\varphi)$ on $M$ and $\cK(\Sigma)$ denotes the vector space of Killing vector fields of $\Sigma$.
\end{theorem}

\noindent
In particular, the previous result implies that the infinitesimal deformations of manifolds of type $S^1\times S^3$ are in general obstructed. The reader is referred to Theorem \ref{thm:infbijection} for more details. The Heterotic solitons obtained by imposing $\varphi = \alpha$ are all locally isomorphic to a supersymmetric solution, as a direct inspection of the classification presented in \cite{GFRST} shows. In order to obtain Heterotic solitons not locally isomorphic to a supersymmetric solution we consider instead Heterotic solitons such that $\varphi = 0$ (that is, the dilaton vanishes) and $\alpha \neq 0$. We obtain a classification result, which we summarize as follows. 

\begin{theorem}
Let $M$ be a compact and oriented four-manifold admitting a non-flat Heterotic soliton $(g,\alpha)$ with $\kappa>0$, vanishing dilaton, and parallel torsion. Then the kernel of $\alpha$ defines an integrable distribution whose leaves, equipped with the metric induced by $g$, are all isometric to an oriented Riemannian three-manifold $(\Sigma,h)$ satisfying one of the following possibilities:	
\begin{enumerate}[leftmargin=*]
\item There exists a double cover of $(\Sigma,h)$ that admits a Sasakian structure determined by $h$ as prescribed in Theorem \ref{thm:solutions}.

\item $(\Sigma,h)$ is isometric to a discrete quotient of either $\widetilde{\mathrm{Sl}}(2,\mathbb{R})$ or $\mathrm{E}(1,1)$ (the universal cover of the Poincar\'e group of two-dimensional Minkowski space) equipped with a left-invariant metric with constant principal Ricci curvatures given by $(0,0,-\frac{1}{2\kappa})$.
		
\item $(\Sigma,h)$ is a hyperbolic three-manifold.
\end{enumerate}
\end{theorem}

\noindent
The reader is referred to Theorem \ref{thm:solutions} for more details and a precise statement of the result. The previous theorem can be used to obtain large families of Heterotic solitons with vanishing dilaton and parallel torsion, as summarized for instance in Corollaries \ref{cor:examples1} and \ref{cor:examplesSasakian}.

Due to the fact that Heterotic solitons constitute a particular class of Heterotic supergravity solutions, they are expected to inherit a \emph{generalized geometric} interpretation on a transitive Courant algebroid, as described in \cite{Ashmore:2019rkx,Coimbra:2014qaa,Garcia-Fernandez:2013gja,GFRTGauge} for the general bosonic sector of Heterotic supergravity. Adapting the framework developed in Op. Cit. to Heterotic solitons would yield a natural geometric framework, adapted to the symmetries of the system, to further investigate Heterotic solitons and their moduli. The power of this formalism is illustrated in \cite{StreetsYuryII,StreetsYury}, where generalized Ricci solitons were thoroughly studied in the framework of generalized complex geometry. The generalized geometry underlying Heterotic supergravity is also positioned to play a prominent role in the study of the T-duality \cite{Bergshoeff:1995cg,Baraglia:2013wua,Garcia-Fernandez:2016ofz} of Heterotic solitons, which is a fundamental tool to classify the latter and to generate new Heterotic solitons of novel topologies. In this context, a specially attractive case corresponds to considering left-invariant Heterotic solitons on four-dimensional Lie groups, where T-duality can be algebraically described \cite{delBarco:2017qol}. We plan to develop these ideas in future publications.


\section{Four-dimensional Heterotic supergravity}
\label{sec:Heterotic4d}



\subsection{Preliminaries}
\label{subsec:preliminaries}

 
Let $M$ be an oriented four-dimensional manifold and let $P$ be a principal bundle over $M$ with semi-simple and compact structure group $\G$. Denote by $\mathfrak{g}$ the Lie algebra of $\G$. We fix an invariant and positive-definite symmetric bilinear form $c\colon \mathfrak{g}\times\mathfrak{g}\to \mathbb{R}$ on $\mathfrak{g}$, and we denote by $\mathfrak{c}$ the inner product induced by $c$ on the adjoint bundle $\mathfrak{g}_P := P\times_{\Ad}\mathfrak{g}$ of $P$. We denote by $\cA_P$ the affine space of connections on $P$ and for every connection $A\in \cA_P$ we denote by $\cF_A\in \Omega^2(\frg_P)$ its curvature. For every Riemannian metric $g$ on $M$, we denote by $\mF_g(M)$ the bundle of oriented orthonormal frames defined by $g$ and the given orientation of $M$, and we denote by $\mathfrak{so}_g(M) :=  \mF_g(M)\times_{\Ad}\mathfrak{so}(4)$ its associated adjoint bundle of $\mathfrak{so}(4)$ algebras, which we will consider equipped with the positive-definite inner product $\mathfrak{v}$ yielded by the trace in $\mathfrak{so}(4)$. The curvature of a connection $\nabla$ on $\mF_g(M)$ will be denoted by $\mR_{\nabla}\in \Omega^2(\mathfrak{so}_g(M))$. Given $(M,P,\mathfrak{c})$ and a Riemannian metric $g$ on $M$, we define the following bilinear map:
\begin{equation*}
\mathfrak{c}(- \circ -)\colon \Omega^k(\mathfrak{g}_P)\times \Omega^k(\mathfrak{g}_P) \to \Gamma(T^{\ast}M\odot T^{\ast}M)\, ,
\end{equation*} 

\noindent
as follows:
\begin{equation*}
\mathfrak{c}(\alpha\circ \beta)(v_1,v_2) \eqdef \frac12((g\otimes \mathfrak{c})(v_1\lrcorner \alpha, v_2\lrcorner \beta) +(g\otimes \mathfrak{c})(v_2\lrcorner \alpha, v_1\lrcorner \beta))\, , 
\end{equation*}

\noindent
for every pair of vector fields $v_1,v_2\in \mathfrak{X}(M)$ and any pair of $k$-forms $\alpha, \beta \in \Omega^k(\mathfrak{g}_P)$ taking values in $\mathfrak{g}_P$. Here $g\otimes\mathfrak{c}(-,-)$ denotes the non-degenerate metric induced by $g$ and $\mathfrak{c}$ on the differentiable forms valued in $\mathfrak{g}_P$. In particular, for the curvature $\cF_A\in \Omega^2(\mathfrak{g}_P)$ of a connection $A\in\cA_P$ we have:
\begin{equation*}
\mathfrak{c}(\cF_A\circ \cF_A)(v_1,v_2) \eqdef (g\otimes \mathfrak{c})(v_1\lrcorner \cF_A, v_2\lrcorner \cF_A)\, , \qquad v_1, v_2 \in \mathfrak{X}(M)\, ,
\end{equation*} 
where $v_1\lrcorner \cF_A$ denotes the 1-form with values in $\mathfrak{g}_P$ obtained by evaluation of $v_1$ in $\cF_A$, and similarly for $v_2\lrcorner \cF_A$. If $\left\{ T_a \right\}$ denotes a local orthonormal frame on $\mathfrak{g}_P$ satisfying $\mathfrak{c}(T_a,T_b)=\delta_{ab}$ and $e_i$ denotes a local orthonormal frame of $(TM,g)$, then the expression above reads:
\begin{equation*}
\mathfrak{c}(\cF_A\circ \cF_A)(v_1,v_2) = \sum_{a,i}\cF_A^a(v_1,e_i)\, \cF_A^a(v_2,e_i)\, .
\end{equation*} 

\noindent
Therefore, in local coordinates $\left\{ x^i\right\}$, $i,j,k,m = 1,\hdots,4$, the previous equation corresponds to:
\begin{equation*}
\mathfrak{c}(\cF_A\circ \cF_A)(v_1,v_2) = \sum_{a} (\cF_A^a)_{im}\, (\cF_A^a)_{jk} \, g^{mk}\, .
\end{equation*} 

\noindent
Similarly, for a 3-form $H\in\Omega^3(M)$ we define:
\begin{equation*}
(H\circ H)(v_1,v_2) \eqdef g(v_1\lrcorner H, v_2 \lrcorner H)\, , \qquad v_1, v_2 \in \mathfrak{X}(M)\, ,
\end{equation*}

\noindent
which in local coordinates reads:
\begin{equation*}
(H\circ H)_{ij} = H_{ilm} H_{j}^{\,\,\, lm}\, .
\end{equation*}

\noindent
Note that the inner product induced by $g$ is to be understood in the sense of tensors (rather than forms). The analogous bilinear map:
\begin{equation*}
	\mathfrak{v}(- \circ -)\colon \Omega^k(\mathfrak{so}_g(M))\times \Omega^k(\mathfrak{so}_g(M)) \to \Gamma(T^{\ast}M\odot T^{\ast}M)\, ,
\end{equation*} 

\noindent
is defined identically to $\mathfrak{c}(- \circ -)$. In particular, in local coordinates we have:
\begin{equation*}
	\mathfrak{v}(\mR_{\nabla}\circ\mR_{\nabla})_{ij} = (\mR_{\nabla})_{iklm}(\mR_{\nabla})^{\,\,klm}_j  \, ,
\end{equation*} 

\noindent 
where $(\mR_{\nabla})_{iklm}$ is the local coordinate expression of the curvature tensor of the connection $\nabla$ on $\mF_g(M)$.

\begin{remark}
For any Riemannian metric $g$ and 3-form $H$ on $M$ we define the connection $\nabla^H$ on the tangent bundle $TM$ as the unique $g$-compatible connection on $M$ with totally antisymmetric torsion given by $- H$. The metric connection $\nabla^H$ is explicitly given in terms of the Levi-Civita connection $\nabla^g$ associated to $g$ as follows:
\begin{equation*}
\nabla^{H} = \nabla^g - \frac{1}{2} H^{\sharp}\, ,
\end{equation*}
	
\noindent
where:
\begin{equation*}
H^{\sharp}(v_1,v_2) = H(v_1,v_2)^{\sharp} = (v_2\lrcorner v_1 \lrcorner H)^{\sharp}\in TM\, , \qquad \forall\,\, v_1 , v_2 \in TM\, , 
\end{equation*}
	
\noindent
and $\sharp\colon T^{\ast}M \to TM$ is the musical isomorphism induced by $g$. 
\end{remark}


\subsection{The equations of motion}
\label{subsec:eqsofmotion}


 We introduce the bosonic sector of Heterotic supergravity through its equations of motion, which also admit a local lagrangian formulation that will not be relevant for our purposes.

\begin{definition}
\label{def:HeteroticSugra}
Let $\kappa > 0$ be a positive real constant. The bosonic sector of {\bf Heterotic supergravity} on $(M,P,\mathfrak{c})$ is defined through the following system of partial differential equations \cite{BRI,BRII,GFRSTII}:
\begin{eqnarray}
\label{eq:motionHetsugra}
& \operatorname{Ric}^{g} +  \nabla^{g}\varphi - \frac{1}{4} H \circ H - \kappa\, \mathfrak{c}(\cF_A \circ \cF_A) + \kappa\, \mathfrak{v}(\mR_{\nabla^{H}} \circ\mR_{\nabla^{H}}) = 0\, ,\nonumber\\
&\delta^g H +  \iota_{\varphi} H = 0\, , \quad \dd_A \ast \cF_A - \varphi\wedge \ast \cF_A  - \cF_A \wedge \ast H = 0\, , \\
&\delta^g\varphi + |\varphi|^2_g  - |H|^2_g -  \kappa\, |\cF_A|^2_{g,\mathfrak{c}} +  \kappa\, |\mR_{\nabla^{H}}|^2_{g,\mathfrak{v}} = 0\nonumber\, ,
\end{eqnarray}

\noindent
together with the \emph{Bianchi identity}:

\begin{equation}
\label{eq:BianchiT}
\dd H = \kappa (\mathfrak{c}\left( \cF_A \wedge \cF_A\right) - \mathfrak{v}(\mR_{\nabla^{H}}\wedge\mR_{\nabla^{H}} ))\, ,
\end{equation}

\noindent
for tuples $(g,H,\varphi,A)$, where $g$ is a Riemannian metric on $M$, $\varphi\in\Omega^{1}_{cl}(M)$ is a closed one form, $H\in \Omega^{3}(M)$ is a 3-form and $A\in \cA_P$ is a connection on $P$. Here the Hodge dual $\ast$ is defined with respect to $g$ and the induced Riemannian volume form.
\end{definition}

\noindent
The norms $\vert -\vert_g$, $\vert -\vert_{g,\frc}$ and $\vert -\vert_{g,\frv}$ are all taken as norms on forms by interpreting the curvatures $\cF_A$ and $\mR_{\nabla^{H}}$ as 2-forms taking values on the adjoint bundle of $P$ and $\mathrm{F}_g(M)$, respectively. This convention is delicate for $\mR_{\nabla^{H}}\in \Omega^2(\mathfrak{so}_g(M))$. In this case, $\mathfrak{so}_g(M) \subset \End(TM)$ is naturally isomorphic to $\Lambda^2 T^{\ast}M$ and $\mR_{\nabla^{H}}$ can be interpreted as a section of $\Lambda^2 T^{\ast}M\otimes \Lambda^2 T^{\ast}M$. Within this interpretation, the norm induced by $\frv$ is by definition the form norm in the first factor $\Lambda^2 T^{\ast}M$ and the tensor norm in the second factor $\Lambda^2 T^{\ast}M = \mathfrak{so}_g(M)$. Hence:
\begin{equation*}
|\mR_{\nabla^{H}}|^2_{g,\mathfrak{v}} = \frac{1}{2} \mathrm{Tr}_g (\mathfrak{v}(\mR_{\nabla^{H}} \circ\mR_{\nabla^{H}}))\, ,
\end{equation*}

\noindent
and, in local coordinates:
\begin{equation*}
|\mR_{\nabla^{H}}|^2_{g,\mathfrak{v}} = \frac{1}{2} (\mR_{\nabla^{H}})_{ijkl} (\mR_{\nabla^{H}})^{ijkl}	\, .
\end{equation*}

\noindent
Alternatively, and as mentioned earlier, $\frv$ can be defined as the norm induced by the form norm on 2-forms and the trace norm for elements in $\mathfrak{so}_g(M) \subset \End(TM)$.

\begin{remark}
Equations \eqref{eq:motionHetsugra} and \eqref{eq:BianchiT} are completely and unambiguously determined by supersymmetry, see for instance \cite{Ortin} and references therein for more details. In particular, these equations describe the low-energy dynamics of the massless bosonic sector of Heterotic string theory. The first equation in \eqref{eq:motionHetsugra} is usually called the \emph{Einstein equation}, the second equation in \eqref{eq:motionHetsugra} is usually called the \emph{Maxwell equation}, the third equation in \eqref{eq:motionHetsugra} is usually called the \emph{Yang-Mills equation} whereas the last equation in \eqref{eq:motionHetsugra} is usually called the \emph{dilaton equation}. The constant $\kappa$ is the \emph{string slope} parameter and has a specific physical interpretation which is not relevant for our purposes.
\end{remark}

\noindent
Suppose that $M$ admits spin structures. Given a tuple $(g,\varphi,H,A)$ as introduced above and a choice of $\mathrm{Spin}(4)$ structure $Q_g$, we denote by $\s_g$ the bundle of irreducible complex spinors canonically associated to $Q_g$. This is a rank-four complex vector bundle $\s_g$ which admits a direct sum decomposition:
\begin{equation*}
\s_g = \s^+_g \oplus \s^-_g\, , \qquad \s^{\pm}_g := \frac{1}{2}(\mathrm{Id}\mp  \nu_g) \s_g\, ,
\end{equation*}

\noindent
in terms of the rank-two chiral bundles $\s^+_g$ and $\s^-_g$. The symbol $\nu_g$ denotes the Riemannian volume form on $(M,g)$ acting by Clifford multiplication on $\s_g$.

\begin{definition}
We say that a tuple $(g,\varphi,H,A)$ solving Equation \eqref{eq:motionHetsugra} is a \emph{supersymmetric solution} of Heterotic supergravity if there exists a bundle of irreducible complex spinors $\s_g = \s^+_g \oplus \s^-_g$ on $(M,g)$ and a spinor $\epsilon\in \Gamma(\s^+_{g})$ such that the following equations are satisfied:
\begin{equation}
\label{eq:susytransHeterotic}
\nabla^{-H}\epsilon = 0\, , \qquad (\varphi - H)\cdot\epsilon = 0\, , \qquad \cF_A\cdot \epsilon = 0\, .
\end{equation}

\noindent
Equations \eqref{eq:susytransHeterotic} are called the \emph{Killing spinor equations} of Heterotic supergravity. For ease of notation we denote with the same symbol the canonical lift of $\nabla^{-H}$ (which has torsion $H$) to the spinor bundle $\s_g$.
\end{definition}

\begin{remark}
The existence of solutions to equations \eqref{eq:susytransHeterotic} may depend on the choice of spin structure on $M$, in the sense that a supersymmetric solution on $M$ with respect to a particular choice of spin structure may be non-supersymmetric with respect to a different choice of spin structure, see \cite{FigueroaGadhia} for more details and explicit examples of this situation.
\end{remark}

\begin{remark}
By a theorem of S. Ivanov \cite{SIvanov}, a quintuple $(g,\varphi,H,A,\epsilon)$ satisfying the Killing spinor equations and the Bianchi identity automatically satisfies all the equations of motion of Heterotic supergravity if and only if the connection $\nabla^H$ is an \emph{instanton}. 
\end{remark}

\noindent
The existence of Killing spinor equations compatible with the system \eqref{eq:motionHetsugra} and \eqref{eq:BianchiT}, in the sense specified in the previous remark, is a consequence of supersymmetry. More precisely, the Killing spinor equations are obtained by imposing the vanishing of the Heterotic supersymmetry transformations on a given bosonic background. We refer the reader to \cite{Gran:2018ijr,Ortin} and references therein for more details.

There is a large amount of meat to unpack in the partial differential equations that define Heterotic supergravity. In order to proceed further it is convenient to consider a reformulation of Heterotic supergravity that profits from the fact that we restrict the underlying manifold to be four-dimensional. For every tuple $(g,\varphi,H,A)$, we define $\alpha := -\ast H\in \Omega^1(M)$.

\begin{lemma}
\label{lemma:BianchiAnsatz}
A tuple $(g,\varphi,H,A)$ with $H=\ast\alpha$ satisfies the Bianchi identity if and only if:
\begin{equation*}
\frac{1}{\kappa} \delta^g\alpha = |\cF^{-}_A|^2_{g,\mathfrak{c}} - |\mR^{-}_{\nabla^{H}}|^2_{g,\mathfrak{v}} - |\cF^{+}_A|^2_{g,\mathfrak{c}} + |\mR^{+}_{\nabla^{H}}|^2_{g,\mathfrak{v}}\, ,
\end{equation*}
where:
\begin{equation*}
\cF^{+}_A := \frac{1}{2}(\cF_A + \ast \cF_A) \, ,\qquad \cF^{-}_A := \frac{1}{2}(\cF_A - \ast \cF_A) \, ,
\end{equation*}
	
\noindent
respectively denotes the self-dual and anti-self-dual projections of $\cF_A$, and similarly for $\mR^{\pm}_{\nabla^{H}}$.
\end{lemma}

\begin{proof}
Using that $\cF^{+}_A \wedge \cF^{-}_A=0$ and $\mR^{+}_{\nabla^{H}}\wedge \mR^{-}_{\nabla^{H}}=0$, we compute:
\begin{eqnarray*}
&  - \frac{1}{\kappa} \delta^g \alpha =\frac{1}{\kappa} \ast \dd H = \ast\mathfrak{c}(\cF^{+}_A\wedge \cF_A^{+}) + \ast\mathfrak{c}(\cF^{-}_A\wedge \cF_A^{-}) - \ast\mathfrak{v}( \mR^{+}_{\nabla^{H}}\wedge \mR^{+}_{\nabla^{H}} ) -\ast\mathfrak{v}( \mR^{-}_{\nabla^{H}}\wedge \mR^{-}_{\nabla^{H}} ) 
\\ 
& = \ast\mathfrak{c}(\cF^{+}_A\wedge \ast \cF_A^{+}) - \ast\mathfrak{c}(\cF^{-}_A\wedge \ast \cF_A^{-}) - \ast\mathfrak{v}( \mR^{+}_{\nabla^{H}}\wedge \ast \mR^{+}_{\nabla^{H}} ) + \ast\mathfrak{v}( \mR^{-}_{\nabla^{H}}\wedge \ast \mR^{-}_{\nabla^{H}} )\\ 
& = \vert \cF_A^{+}\vert^2_{g,\mathfrak{c}} - \vert \cF_A^{-}\vert^2_{g,\mathfrak{c}} - |\mR^{+}_{\nabla^{H}}|^2_{g,\mathfrak{v}} + |\mR^{-}_{\nabla^{H}}|^2_{g,\mathfrak{v}}\, ,
\end{eqnarray*}
and hence we conclude.
\end{proof}

\noindent
On the other hand, regarding the Maxwell equation in \eqref{eq:motionHetsugra} we have:
\begin{equation*}
\delta^g H + \iota_{\varphi}H = \ast \dd \alpha +\iota_{\varphi}\ast\alpha =\ast (\dd \alpha -\varphi \wedge \alpha) = 0\, ,
\end{equation*}

\noindent
whence it is equivalent to $\dd \alpha = \varphi\wedge \alpha$. The previous computation together with Lemma \ref{lemma:BianchiAnsatz} proves that four-dimensional Heterotic supergravity, as introduced in Definition \ref{def:HeteroticSugra}, can be equivalently written as follows:

\begin{eqnarray}
\label{eq:motionHetsugrabox1}
&\mathrm{Ric}^{g} +  \nabla^{g}\varphi + \frac{1}{2} \alpha\otimes \alpha - \frac{1}{2} \vert \alpha \vert^2_g\, g + \kappa \,\mathfrak{v}(\mR_{\nabla^{\alpha}} \circ\mR_{\nabla^{\alpha}}) = \kappa\, \mathfrak{c}(\cF_A \circ \cF_A)\, , \quad \dd\alpha = \varphi\wedge \alpha\\ 
\label{eq:motionHetsugrabox2}
& \dd_A^{\ast}  \cF_A + \iota_{\varphi}\cF_A -  \iota_{\alpha} \ast \cF_A = 0 \, , \quad \delta^g \varphi + |\varphi|^2_g + \kappa |\mR_{\nabla^{\alpha}}|^2_{g,\mathfrak{v}}   = |\alpha|^2_g +  \kappa |\cF_A|^2_{g,\mathfrak{c}}\, , \\
\label{eq:BianchiTbox}
&\frac{1}{\kappa} \delta^g\alpha = |\cF^{-}_A|^2_{g,\mathfrak{c}} - |\mR^{-}_{\nabla^{\alpha}}|^2_{g,\mathfrak{v}} - |\cF^{+}_A|^2_{g,\mathfrak{c}} + |\mR^{+}_{\nabla^{\alpha}}|^2_{g,\mathfrak{v}}
\end{eqnarray}

\noindent
for tuples $(g,\varphi,\alpha,A)$, where by definition we have set $\nabla^{\alpha} := \nabla^{H}$ with $H = \ast \alpha$. To every solution $(g,\varphi,\alpha,A) $ of Heterotic supergravity we can associate a cohomology class $\sigma$ in $H^1(M,\mathbb{R})$ defined by $\sigma := [\varphi]\in H^1(M,\mathbb{R})$. We will call $\sigma$ the \emph{Lee class} of $(g,\varphi,\alpha,A)$. If a solution exists, the Bianchi identity immediately implies the following equation in $H^4(M,\mathbb{R})$:
\begin{equation*}
p_1(P) = p_1(M)\in H^4(M,\mathbb{R})\, ,
\end{equation*}

\noindent
that is, the first Pontryagin  class $p_1(P)$ of $P$ needs to be equal to the first Pontryagin class $p_1(M)$ of $M$ with real coefficients. This gives a simple topological obstruction to the existence of Heterotic solutions on a given triple $(M,P,\frc)$.

\begin{remark}
As we will see later, see for instance Section \ref{sec:NS-NSmoduli}, the topology and geometry of compact four-manifolds admitting solutions of Heterotic supergravity depends crucially on whether $\sigma = 0$ or $\sigma \neq 0$.
\end{remark}


\subsection{Relation to other formulations of Heterotic supergravity}


The formulation of the bosonic sector of Heterotic supergravity that we have considered in order to define the system \eqref{eq:motionHetsugra}--\eqref{eq:BianchiT} corresponds to a direct truncation of the Heterotic supergravity constructed in \cite{BRI,BRII}. Within this formulation of Heterotic supergravity, the \emph{higher order} terms of the theory are constructed through contractions of the curvature tensor $\mR_{\nabla^{H}}$ of the metric connection with torsion $\nabla^H$, sometimes called the \emph{Hull connection} \cite{Hull,delaOssa:2014msa}. It is however possible to obtain a consistent theory of Heterotic supergravity for which the higher order terms are constructed through the curvature tensor of a different fixed metric connection. The ambiguities associated with this choice of connection have been extensively discussed in the literature, both from the world-sheet perspective \cite{HullTownsend,Sen}, where the change of such connection can be shown to correspond to a certain field redefinition, and from the supergravity point of view \cite{delaOssa:2014msa,Hull,Melnikov:2014ywa}, where the change of such connection can be proven to correspond to a modification in the regularization scheme of the effective action. Alternatively, and as explained in the introduction, solutions to the differential system \eqref{eq:motionHetsugra}--\eqref{eq:BianchiT} can be understood as self-similar solutions of the renormalization group flow of the NS-NS sector of the Heterotic world-sheet at first order in the string slope parameter \cite{HeteroticRicciFlow}.


\section{Heterotic solitons and the moduli of manifolds of type $S^1\times S^3$}
\label{sec:NS-NSmoduli}


This section introduces the notion of \emph{Heterotic soliton} and develops the classification of NS-NS pairs, introduced below, which will lead us to study the global moduli space of \emph{manifolds of type $S^1\times S^3$} as defined by P. Gauduchon in \cite{Gauduchon}.


\subsection{Heterotic solitons}
\label{sec:NSNSsupergravity}


Assuming that $P$ is the trivial principal bundle over $M$, that is $P=M$, the triple $(M,P,\mathfrak{c})$ reduces to the oriented four-manifold $M$. In this case, the configuration space of four-dimensional Heterotic supergravity, which we denote by $\Conf_{\kappa}(M)$, consists of all triples of the form $(g,\varphi,\alpha)$, where $g$ is a Riemannian metric on $M$, $\varphi$ is a closed 1-form and $\alpha$ is a 1-form.  Four-dimensional Heterotic supergravity reduces to:
\begin{eqnarray}
	\label{eq:HeteroticRicci1}
	&\mathrm{Ric}^{g} +  \nabla^{g}\varphi + \frac{1}{2} \alpha\otimes \alpha - \frac{1}{2} \vert \alpha \vert^2_g\, g + \kappa \,\mathfrak{v}(\mR_{\nabla^{\alpha}} \circ\mR_{\nabla^{\alpha}}) = 0\, , \quad \dd\alpha = \varphi\wedge \alpha\\ 
	& \delta^g \varphi + |\varphi|^2_g + \kappa \,|\mR_{\nabla^{\alpha}}|^2_{g,\mathfrak{v}}   = |\alpha|^2_g \, , \qquad \delta^g \alpha = \kappa \,(|\mR^{+}_{\nabla^{\alpha}}|^2_{g,\mathfrak{v}}  - |\mR^{-}_{\nabla^{\alpha}}|^2_{g,\mathfrak{v}})\, ,
	\label{eq:HeteroticRicci2}
\end{eqnarray}

\noindent
for $(g,\varphi,\alpha)\in \Conf_{\kappa}(M)$. In the limit $\kappa \to 0$, the previous system recovers the \emph{generalized} Ricci soliton system in four dimensions \cite{FernandezStreets} and therefore can be considered as a natural generalization of the latter in the context of Heterotic string theory corrections to the effective supergravity action. We introduce now the following definition.
\begin{definition}
\label{def:HeteroticSoliton}
The (four-dimensional) \emph{Heterotic soliton system} consists of equations \eqref{eq:HeteroticRicci1} and \eqref{eq:HeteroticRicci2}. Solutions of the Heterotic soliton system are (four-dimensional) \emph{Heterotic solitons}.
\end{definition}

\noindent
If we further impose $\alpha =\varphi$ the Heterotic soliton system \eqref{eq:HeteroticRicci1}--\eqref{eq:HeteroticRicci2} further reduces to:
\begin{eqnarray}
	\label{eq:motionNSNS1}
	&\mathrm{Ric}^{g} +  \nabla^{g}\varphi + \frac{1}{2} \varphi\otimes \varphi - \frac{1}{2} \vert \varphi \vert^2_g\, g + \kappa\,\mathfrak{v}(\mR_{\nabla^{\varphi}} \circ\mR_{\nabla^{\varphi}}) = 0\, , \\ 
	\label{eq:motionNSNS2}
	& \delta^g\varphi + \kappa\,|\mR^{-}_{\nabla^{\varphi}}|^2_{g,\mathfrak{v}} = 0 \, , \qquad |\mR^{+}_{\nabla^{\varphi}}|^2_{g,\mathfrak{v}} = 0\, ,
\end{eqnarray}

\noindent
for pairs $(g,\varphi)$ consisting on a Riemannian metric $g$ on $M$ and a closed 1-form $\varphi\in \Omega^1_{cl}(M)$. Equations \eqref{eq:motionNSNS1} and \eqref{eq:motionNSNS2} define, in physics terminology, the so-called \emph{NS-NS supergravity}. Consequently, we will refer to pairs $(g,\varphi)$ solving \eqref{eq:motionNSNS1} and \eqref{eq:motionNSNS2} as \emph{NS-NS pairs}.


\subsection{Compact NS-NS pairs}
\label{sec:compactNSNSpairs}


Let $M$ be an oriented and connected four-manifold equipped with a NS-NS pair $(g,\varphi)$. Recall that the connection $\nabla^{\varphi}$ is an anti-self-dual instanton on the tangent bundle of $M$. We will say that a NS-NS pair is \emph{complete} if $(M,g)$ is a complete Riemannian four-manifold.
 
\begin{lemma}
Let $(g,\varphi)$ be a NS-NS pair on $M$. We have:
\begin{equation*}
\mathfrak{v}(\mR_{\nabla^{\varphi}} \circ\mR_{\nabla^{\varphi}}) = \frac{g}{2}|\mR^{-}_{\nabla^{\varphi}}|^2_{g,\mathfrak{v}}  \, ,
\end{equation*}

\noindent
and $(g,\varphi)$ satisfies:
\begin{equation}
\label{eq:conformalEinsteinNS}
\mathrm{Ric}^{g} +  \nabla^{g}\varphi + \frac{1}{2} \varphi\otimes \varphi - \frac{1}{2}( \vert \varphi \vert^2_g  +   \delta^g \varphi) g = 0\, ,
\end{equation}

\noindent
which is equivalent to the Einstein equation \eqref{eq:motionNSNS1} for $(g,\varphi)$.
\end{lemma}

\begin{proof}
Let $T_a$ denote a local basis of $\mathfrak{so}_g(M)$ satisfying $\mathfrak{v}(T_a,T_b) = \delta_{ab}$ and write $\mR_{\nabla^{\alpha}} = \sum_a\mR^a_{\nabla^{\alpha}}\otimes T_a$. Identifying each 2-form $\mR_{\nabla^{\alpha}}^a$ with a skew-symmetric endomorphism of $TM$ we have:
\begin{equation*}
	\mathfrak{v}(\mR_{\nabla^{\alpha}}\circ \mR_{\nabla^{\alpha}})(v_1,v_2) =  \sum_{a} g(\mR_{\nabla^{\alpha}}^a\circ \mR_{\nabla^{\alpha}}^a(v_1),v_2).
\end{equation*} 

\noindent
Using that $\mR_{\nabla^{\alpha}}$ is anti-self-dual, the same holds for each component $\mR_{\nabla^{\alpha}}^a$, and thus $\mR_{\nabla^{\alpha}}^a\circ \mR_{\nabla^{\alpha}}^a= \frac12|\mR_{\nabla^{\alpha}}^a|^2\mathrm{Id}_{TM}$. Hence, we obtain: 
\begin{equation*}
\mathfrak{v}(\mR_{\nabla^{\alpha}}\circ\mR_{\nabla^{\alpha}}) = \frac12|\mR_{\nabla^{\alpha}}|^2_{g,\mathfrak{v}}g\, .
\end{equation*} 
\noindent
The second part follows directly after substituting the first equation in \eqref{eq:motionNSNS2} into equation \eqref{eq:motionNSNS1}, upon use of the previous identity.
\end{proof}

\noindent
Equation \eqref{eq:conformalEinsteinNS} can be naturally interpreted in the framework of conformal geometry and Einstein-Weyl structures. Let $\cC$ be the conformal class of Riemannian metrics on $M$ containing $g$, and assume that  
\begin{equation*}
D g = -2\theta\otimes g\, .
\end{equation*}

\noindent
The Ricci curvature $\mathrm{Ric}^D$ of $D$ reads:
\begin{equation}
	\label{eq:Ricci1}
	\mathrm{Ric}^D=\mathrm{Ric}^g-2(\nabla^g\theta-\theta\otimes\theta)+(\delta^g\theta-2|\theta|_g^2)g\, .
\end{equation}

\noindent
Using the previous expression, we readily conclude that \eqref{eq:conformalEinsteinNS} is equivalent to $\mathrm{Ric}^{D}=0$, where $D$ is the Weyl connection on $(M,\cC)$ whose Lee form with respect to $g$ is  $\theta=-\frac\varphi{2}$. Consequently \eqref{eq:conformalEinsteinNS} is conformally invariant, in the sense that, given a NS-NS pair $(g,\varphi)$, every other metric $\tilde g=e^{f}g$ in the conformal class of $g$ satisfies: 
\begin{equation}
	\label{eq:conformalEinsteinNStilde}
	\operatorname{Ric}^{\tilde g} +  \nabla^{\tilde g}\tilde \varphi + \frac{1}{2}\tilde  \varphi\otimes\tilde \varphi-\frac12(|\tilde \varphi|_{\tilde g}^2+\delta^{\tilde g}\tilde \varphi)\tilde g = 0\, ,
\end{equation}

\noindent
for $\tilde\varphi:=\varphi+\dd f$. Recall that a closed Weyl structure is said to be \emph{closed Einstein-Weyl} if it satisfies \eqref{eq:conformalEinsteinNStilde}.

\begin{lemma}
\label{lemma:thetaparallel}
Let $(\cC,D)$ a closed Einstein-Weyl structure on a compact four-manifold $M$. Then, the Lee-form $\theta$ associated to the Gauduchon metric $g$ of $\cC$ is parallel.
\end{lemma}

\begin{proof}
Let $(\cC,D)$ a closed Einstein-Weyl manifold and let $(g,\theta)$ be a Gauduchon representative, that is, $\theta$ is coclosed with respect to $g$. Hence, the pair $(g,\theta)$ satisfies: 
\begin{equation*}
\mathrm{Ric}^g-2(\nabla^g\theta-\theta\otimes\theta) -2|\theta|_g^2 g = 0\, .
\end{equation*}

Taking the trace in this equation and using the fact that $\theta$ is coclosed, we obtain that the scalar curvature $s^g$ of $g$ satisfies $s^g=6|\theta|_g^2$. Using the contracted Bianchi identity $(\nabla^g)^{\ast} \mathrm{Ric}^g = -\frac{1}{2} \dd s^g$ and the formula $(\nabla^g)^{\ast} (|\theta|_g^2 g)=-\dd|\theta|_g^2$, we then compute the total norm of $\nabla^g\theta$ with respect to $g$:
\begin{equation*}
	\norm{\nabla^g\theta}^2_g = \langle\nabla^g\theta, \frac{1}{2}\mathrm{Ric}^g +\theta\otimes\theta - |\theta|_g^2 g\rangle_g = \langle \theta, (\nabla^g)^{\ast}(\theta\otimes \theta)+\frac12\dd|\theta|_g^2\rangle_g= \langle \theta, (\nabla^g)^{\ast}(\theta\otimes \theta)\rangle_g\, .
\end{equation*}

On the other hand:
\begin{equation*}
	\langle \theta, (\nabla^g)^{\ast}(\theta\otimes \theta)\rangle_g = -\frac{1}{2} \int_M \langle \theta, \dd|\theta|_g^2\rangle_g \nu_g = 0\, ,
\end{equation*}

\noindent
where $\nu_g$ denotes the Riemannian volume volume form on $(M,g)$. Hence $\nabla^g\theta = 0$.
\end{proof}

\begin{proposition}
\label{prop:compactNSNS}
Assume $M$ is compact and admits a NS-NS pair $(g,\varphi)\in \Sol_{\kappa}(M)$, with associated Lee class $\sigma \in H^1(M,\mathbb{R})$.
\begin{enumerate}[leftmargin=*]
\item If $\sigma = [\varphi] = 0 \in H^1(M,\mathbb{R})$ then $(M,g)$ is flat and therefore admits a finite covering conformal to a flat torus.
		 
\item If $ 0 \neq \sigma = [\varphi] \in H^1(M,\mathbb{R})$, then $b^1(M) = 1$, and the universal Riemannian cover of $(M,g)$ is isometric to $\mathbb{R}\times S^3$ equipped with the direct product of the standard metric of $\mathbb{R}$ and the round metric on $S^3$ of sectional curvature $\frac{1}{4}\vert\varphi\vert^2_g$, where $\vert\varphi\vert_g$ is the point-wise constant norm of $\varphi$.
\end{enumerate}
\end{proposition}

\begin{proof}
If $(g,\varphi)$ is a NS-NS pair with $\sigma = 0$ then $\varphi$ is exact and parallel whence $\varphi =0$ and $(M,g)$ is a flat compact four-manifold, thus finitely covered by a torus. Assume now that $(g,\varphi)$ is a solution with $\sigma \neq 0$. Lemma \ref{lemma:thetaparallel} implies that $\varphi$ is a non-zero parallel 1-form on $M$. Therefore, by the de Rham theorem the universal Riemannian cover $(\hat{M},\hat{g})$ of $(M,g)$ is isometric to a Riemannian product $(\mathbb{R}\times N,\dd t^2 +g_N)$, where $(N,g_N)$ is a complete and simply connected Riemannian three-manifold, such that $\varphi$ is a constant multiple of $\dd t$. The Einstein equation for $(g,\varphi)$ implies that $g_N$ is Einstein with positive sectional curvature $\frac{1}{4}\vert\varphi\vert^2_g$. Therefore, $(N,g_N)$ is isometric to the round sphere $S^3$ of constant sectional curvature $\frac{1}{4}\vert\varphi\vert^2_g$. A direct computation shows that the metric connection on $(\mathbb{R}\times N, \hat{g}=\dd t^2+g_N)$ with torsion $\vert \varphi\vert_g \ast_{\hat{g}}\dd t$ is flat and therefore both equations in \eqref{eq:motionNSNS2} are automatically satisfied and we conclude.
\end{proof}
 
\noindent
Reference \cite{Gauduchon} gives, using results of \cite{CG1,CG2}, a detailed account of compact Riemannian four-manifolds covered by the Riemannian product $\mathbb{R}\times S^3$. These manifolds were called in Op. Cit. \emph{manifolds of type $S^1\times S^3$}. Manifolds of type $S^1\times S^3$ admit a very explicit description, which we will review in the following. This description will be important in order to construct globally the moduli space of NS-NS pairs.


\subsection{Moduli space of manifolds of type $S^1\times S^3$}
\label{sec:ModuliS1S3}


In this Section we construct the global moduli space of NS-NS pairs with non-vanishing Lee class. This is possible due to the fact that, as described in Proposition \ref{prop:compactNSNS}, NS-NS pairs with non-trivial Lee class yield closed Einstein-Weyl structures on $M$. The deformation problem (around an Einstein metric) of Einstein-Weyl structures with Gauduchon constant one has been studied in \cite{Pedersen}. However, the analysis of Op. Cit. does not cover the case we consider here, since the Weyl structure associated to a NS-NS pair $(g,\varphi)$ on $M$ has zero Gauduchon constant and furthermore such $M$, which corresponds to a manifold of type $S^1\times S^3$, does not admit positive curvature Einstein metrics. Concerning the case of vanishing Gauduchon constant, \cite[Remark 7]{Pedersen} states that the moduli space of manifolds of type $S^1\times S^3$ is one-dimensional. We will show in Theorem \ref{thm:bijection} that this is not correct, see also Corollary \ref{cor:NSNSpairs}.

\begin{definition}
\cite{Gauduchon}
\label{def:Gauduchon}
A {\bf manifold of type} $S^1\times S^3$ is a connected and oriented Riemannian manifold locally isometric to $\mathbb{R}\times S^3$, where $\mathbb{R}$ is equipped with its canonical metric and $S^3$ is equipped with its round metric of sectional curvature $\frac14$.
\end{definition}

\begin{remark}
Reference \cite{Gauduchon} introduces manifolds of type $S^1\times S^3$ by requiring the metric on $S^3$ to have sectional curvature 1. Our choice for the sectional curvature to be equal to $\frac14$ in the above definition is motivated by the the fact that in this way NS-NS pairs correspond directly to manifolds of type $S^1\times S^3$, without the need of rescaling the metric.
\end{remark}

\noindent
From its very definition it follows that the universal Riemannian cover of a manifold of type $S^1\times S^3$ is $\mathbb{R}\times S^3$, which we consider to be oriented and \emph{time oriented}, the latter meaning that an orientation on the factor $\mathbb{R}$ has been fixed. Manifolds of type $S^1\times S^3$ are determined by the embedding, modulo conjugation, of their fundamental group $\Gamma$ into the orientation-preserving isometry group $\mathrm{Iso}(\mathbb{R}\times S^3)$ of $\mathbb{R}\times S^3$. Since $\Gamma$ acts without fixed points, we actually have $\Gamma \subset \mathrm{Iso}(\mathbb{R})\times \mathrm{Iso}(S^3)$, that is, elements of $\Gamma$ act by translations on $\mathbb{R}$ preserving the canonical 1-form on $\mathbb{R}$ as well as the orientation on $S^3$. Every manifold $(M,g)$ of type $S^1\times S^3$ can be written as a quotient:
\begin{equation*}
(M,g) = (\mathbb{R}\times S^3)/\Gamma\, ,
\end{equation*}

\noindent
where $\Gamma \subset \mathrm{Iso}(\mathbb{R})\times \mathrm{Iso}(S^3)$ acts freely and properly on $\mathbb{R}\times S^3$ through the action of the isometry group of the latter. Elements of $\mathrm{Iso}(\mathbb{R})\times \mathrm{Iso}(S^3)$ preserve the canonical unit norm vector field on $\mathbb{R}$. Consequently, every manifold of type $S^1\times S^3$ is equipped with a canonical unit norm parallel vector field, whose musical dual corresponds with $\varphi$, modulo a multiplicative positive constant. Alternatively, every manifold of type $S^1\times S^3$ can be obtained from a direct product $[0,a]\times \Sigma$, where $a>0$ is a real constant and $\Sigma$ is a compact Riemannian three-manifold of constant sectional curvature equal to $\frac14$, through the suspension of $\Sigma$ over $[0,a]$ by an isometry $\psi$ of $\Sigma$. Therefore, a manifold of type $S^1\times S^3$ is the total space of a fibration over the circle of length $a$ with fiber $\Sigma$ which comes equipped with a connection of holonomy generated by $\psi$. Each fiber is isometric to a quotient $\Sigma = S^3/\Gamma_0$, where $\Gamma_0$ is a finite subgroup $\Gamma_0 \subset \mathrm{Iso}(S^3) =  \mathrm{SO}(4)$ acting freely on $S^3$ as an embedded subgroup of $\Gamma$ which preserves each sphere $\left\{ t\right\}\times S^3$ in $\mathbb{R}\times S^3$. Therefore, the group of isometries $\mathrm{Iso}(\Sigma)$ is identified canonically with $N(\Gamma_0)/\Gamma_0$, where $N(\Gamma_0)$ is the normalizer of $\Gamma_0$ in  $\mathrm{SO}(4)$. Consequently, the fundamental group of a manifold of type $S^1\times S^3$ is a semi-direct product of $\Gamma_0$ with the infinite cyclic group $\mathbb{Z}$ which is realized as a subgroup of $\mathbb{R}\times \mathrm{SO}(4)$ as follows:
\begin{equation}
\label{eq:Zaction}
n\mapsto (n a , [\psi]^n)\, , \qquad \gamma \mapsto (0,\gamma)\, , \qquad \forall\, n\in \mathbb{Z}\, , \  \forall\, \gamma \in \Gamma_0\, ,
\end{equation}

\noindent
where $\psi \in \mathrm{Iso}(\Sigma)$. In particular, given a closed three manifold $\Sigma = S^3/\Gamma_0$, a pair $(\lambda,\psi)$ consisting in a positive real number $\lambda$ and an isometry $\psi$ of $\Sigma$ uniquely determines a manifold of type $S^1\times S^3$ as the quotient:
\begin{equation}
\label{eq:presentationS1S3}
(M,g) = (\mathbb{R}\times \Sigma)/\langle (\lambda, \psi)\rangle\, ,
\end{equation}

\noindent
where $\psi$ is considered as an element of $\mathrm{Iso}(\Sigma) = N(\Gamma_0)/\Gamma_0$ and $\langle (\lambda, \psi)\rangle$ is the infinite cyclic group generated by the isometry $(\lambda,\psi)$ of $\mathbb{R}\times \Sigma$ acting as the translation by $\lambda$ on $\mathbb{R}$ and $\psi$ on $\Sigma$.

\begin{definition}
A manifold of type $S^1\times S^3$ is of \emph{class} $\Sigma$ with respect to $(\lambda,\psi)$ if it is isometric to a quotient of the form \eqref{eq:presentationS1S3}.
\end{definition}


\begin{lemma}
\label{lemma:isometricS1S3}
Let $F\colon (M_1,g_1)\to (M_2,g_2)$ be an isometry between manifolds of type $S^1\times S^3$ and of class $\Sigma$ with respect to $(\lambda_i,\psi_i)$, with $\lambda_i\in \mathbb{R}_{+}$ and $\psi_i\in \mathrm{Iso}(\Sigma)$. Then, $\lambda_1 = \lambda_2$ and:
\begin{equation*}
\mathfrak{f}\circ \psi_1 \circ \mathfrak{f}^{-1} = \psi_2\, ,
\end{equation*}
	
\noindent
for an isometry $\mathfrak{f} \in \mathrm{Iso}(\Sigma)$.
\end{lemma}

\begin{remark}
We recall that, by definition, $\hat{F}\colon \mathbb{R}\times \Sigma \to \mathbb{R} \times \Sigma$ is a covering lift of $F\colon (M_1,g_1)\to (M_2,g_2)$ if it fits into the following commutative diagram equivariantly with respect to deck transformations:
	
\begin{center}
\begin{tikzpicture} 
\label{diag:commutative}
	\matrix (m) [matrix of math nodes,row sep=5em,column sep=6em,minimum width=2em]
		{
			\mathbb{R} \times \Sigma & \mathbb{R} \times \Sigma \\
			(M_1,g_1) & (M_2,g_2) \\};
		\path[-stealth]
		(m-1-1) edge node [left] {$p_1$} (m-2-1)
		edge node [above] {$\hat{F} = \hat{F}_0 \times\mathfrak{f}$} (m-1-2)
		(m-2-1.east|-m-2-2) edge node [below] {}
		node [above] {$F$} (m-2-2)
		(m-1-2) edge node [right] {$p_2$} (m-2-2);
		\end{tikzpicture}
	\end{center}
	
\noindent
where $p_1$ and $p_2$ denote the cover projections and $\mathbb{R}\times \Sigma$ is endowed with the product metric. In particular, $\hat{F}\in\mathrm{Iso}(\mathbb{R} \times\Sigma)$ is an isometry and $\hat{F}_0$ acts by translations.
\end{remark}

\begin{proof}
Since $p_1 \colon \mathbb{R}\times \Sigma \to (M_1,g_1)$ is a covering map and $F$ is a diffeomorphism, the map:
\begin{equation*}
F\circ p_1 \colon \mathbb{R}\times \Sigma \to (M_2,g_2)\, ,
\end{equation*}	
	
\noindent
is also a covering map. Using the fact that covering maps induce injective morphisms at the level of fundamental groups, it follows that $(F\circ p_1)_{\ast}(\pi_1(\Sigma)) \subset \pi_1(M_2)$ and $(p_2)_{\ast}(\pi_1(\Sigma)) \subset \pi_1(M_2)$ are subgroups of $\pi_1(M_2)$ abstractly isomorphic to $\pi_1(\Sigma)$. Since both $(F\circ p_1)_{\ast}(\pi_1(\Sigma))$ and $(p_2)_{\ast}(\pi_1(\Sigma))$ contain all torsion elements of $\pi_1(M_2)$ and are normal subgroups of $\pi_1(M_2)$, we conclude:  
\begin{equation*}
(F\circ p_1)_{\ast}(\pi_1(\Sigma)) =(p_2)_{\ast}(\pi_1(\Sigma))\, ,
\end{equation*}
	
\noindent
in $\pi_1(M_2)$. Therefore, standard covering theory implies that $F\circ p_1$ and $p_2$ are isomorphic covering maps (equivariantly with respect to deck transformations). Hence, there exists a diffeomorphism $\hat{F}\colon \mathbb{R} \times \Sigma \to \mathbb{R}\times \Sigma$ fitting equivariantly in the commutative diagram \ref{diag:commutative}. This map can be shown to be an isometry with respect to the product metric on $\mathbb{R}\times \Sigma$. The fact that $\hat{F}$ is an isometry implies the decomposition $\hat{F} = \hat{F}_0 \times \mathfrak{f}$ where $\hat{F}_0$ acts by constant translations on $\mathbb{R}$. The equivariance of $\hat{F}$ implies in turn:
\begin{equation*}
\hat{F}((r,s)\cdot (\lambda_1,\psi_1)) = \hat{F}((r,s))\cdot (\lambda_2,\psi_2)^n\, ,
\end{equation*}
	
\noindent
where $n$ is a natural number. The fact that $\hat{F}$ is a diffeomorphism together with the fact that the fibers of $p_a$ are torsors over $\langle \lambda_a, \psi_a\rangle$, $a=1,2$, implies that $n=1$, since otherwise $\hat{F}$ would not be surjective. Therefore:
\begin{equation*}
\hat{F}\circ (\lambda_1,\psi_1)\circ \hat{F}^{-1} = (\lambda_2,\psi_2)\, ,
\end{equation*} 
	
\noindent
implying $\lambda_1 = \lambda_2$, as well as:
\begin{equation*}
\mathfrak{f}\circ \psi_1 \circ \mathfrak{f}^{-1} = \psi_2\, .
\end{equation*}

\noindent
Since the lift $\hat{F}$ we have considered is unique modulo conjugation by isometries in $\mathrm{Iso}(\mathbb{R})\times \mathrm{Iso}(\Sigma)$, we conclude.
\end{proof}

\noindent
Fix now an oriented and closed Riemannian three-manifold of the form $\Sigma = S^3/\Gamma_0$ and define the set:
\begin{equation*}
\cI(\Sigma) := \mathrm{Iso}(\Sigma)/\mathrm{Ad}(\mathrm{Iso}(\Sigma))\, .
\end{equation*}

\noindent
to be the set of orbits of the adjoint action $\mathrm{Ad}\colon \mathrm{Iso}(\Sigma)\to \mathrm{Aut}(\mathrm{Iso}(\Sigma))$, that is, the set of conjugacy classes of $\mathrm{Iso}(\Sigma)$. Furthermore, denote by $\mathfrak{M}(\Sigma)$ the set of manifolds of type $S^1\times S^3$ and of class $\Sigma$ modulo the natural action of the orientation-preserving diffeomorphism group via pull-back.  

\begin{theorem}
\label{thm:bijection}
There is a canonical bijection of sets:
\begin{equation*}
\mathbb{R}_{+}\times \cI(\Sigma) \xrightarrow{\simeq} \mathfrak{M}(\Sigma)\, .
\end{equation*}
\end{theorem}

\begin{proof}
To every element $(\lambda, [\psi])\in \mathbb{R}_{+}\times \cI(\Sigma)$ we associate the element in $\mathfrak{M}(\Sigma)$ given by the isomorphism class of manifolds of type $S^1\times S^3$ defined by the following manifold of type $S^1\times S^3$:
\begin{equation*}
(M,g) = (\mathbb{R}\times \Sigma)/\langle( \lambda, \psi)\rangle\, ,
\end{equation*}
	
\noindent
where $\psi$ is any representative of $[\psi]\in \cI(\Sigma)$. Changing the representative yields an isometric manifold of type $S^1\times S^3$ and class $\Sigma$, whence the assignment is well defined. Conversely, Lemma \ref{lemma:isometricS1S3} implies that to any isomorphism class in $\mathfrak{M}(\Sigma)$ we can associate a unique element in $\mathbb{R}_{+}\times \cI(\Sigma)$ and that this assignment is inverse to the previous construction and thus we conclude.
\end{proof}

\noindent
The set of conjugacy classes of a compact Lie group admits a very explicit description as a polytope in the Cartan algebra of $\mathrm{Iso}(\Sigma)$. Fix a maximal torus $T\subset \mathrm{Iso}(\Sigma)$, with Lie algebra $\mathfrak{t}$. We denote by:
\begin{equation*}
W(\Sigma,T) := \frac{N(T)}{T}\, ,
\end{equation*}

\noindent
the Weyl group of $\mathrm{Iso}(\Sigma)$, where $N(T)$ denotes the normalizer of $T$ in $\mathrm{Iso}(\Sigma)$. The exponential map $\mathrm{Exp}\colon \mathfrak{t}\to T$ gives a surjective map onto $T$ and its kernel is a lattice in $\mathfrak{t}$ which allows to recover $T$ as:
\begin{equation*}
T = \frac{\mathfrak{t}}{\mathrm{ker}(\mathrm{Exp})}\, .
\end{equation*}

\noindent
Every conjugacy class in $\mathrm{Iso}(\Sigma)$ intersects $T$ in at least one point \cite{Hall}, unique modulo the natural adjoint action of the Weyl group $W$ on $T$. This fact can be used to prove that we have a bijection:
\begin{equation*}
\cI(\Sigma) = \frac{T}{W(\Sigma,T)} = \frac{\mathfrak{t}}{W(\Sigma,T)\ltimes \mathrm{ker}(\mathrm{Exp})}\, ,
\end{equation*}

\noindent
which gives an explicit description of $\cI(\Sigma)$ in terms of the fundamental region of the action of $W(\Sigma,T)\ltimes \mathrm{ker}(\mathrm{Exp})$ on $\mathfrak{t}$.

\begin{remark}
The isometry groups of compact elliptic three-manifolds $\Sigma = S^3/\Gamma_0$ have been classified in \cite{McCullough}. The Weyl group of most of the subgroups of $\mathrm{SO}(4)$ appearing as isometry groups of elliptic three-manifolds can be directly computed, a fact that allows for a direct construction of the corresponding moduli space of manifolds of type $S^1\times S^3$.
\end{remark}

\noindent
Let $\mathrm{rk}(\mathrm{Iso}(\Sigma))$ denote the rank of $\mathrm{Iso}(\Sigma)$, that is, the dimension of any of its maximal torus subgroups. As a direct consequence of Theorem \ref{thm:bijection} we obtain the following result.

\begin{corollary}
The moduli space of manifolds of type $S^1\times S^3$ of class $\Sigma$ has dimension $1 + \mathrm{rk}(\mathrm{Iso}(\Sigma))$. 
\end{corollary}

\noindent
Returning to the problem of classifying NS-NS pairs, the previous discussion implies the following classification result.

\begin{corollary}
\label{cor:NSNSpairs}
The moduli space $\mathfrak{M}_{\mathrm{NS}}(\Sigma)$ of NS-NS pairs on a manifold of the form \eqref{eq:presentationS1S3} admits a finite covering given by $\mathbb{R}^2\times T$, where $T$ is a maximal torus of $\mathrm{Iso}(\Sigma)$. In particular $\dim(\mathfrak{M}_{\mathrm{NS}}(\Sigma)) = 2 + \mathrm{rk}(\mathrm{Iso}(\Sigma))$.  
\end{corollary}

\begin{proof}
Every NS-NS pair $(g,\varphi)$ defines a manifold of type $S^1\times S^3$ given by $(\vert\varphi\vert^{2}_g\,g,\vert\varphi\vert^{-2}_g\varphi)$. Indeed, note that $\vert\varphi\vert^{-2}_g\varphi$ has norm one with respect to the metric $\vert\varphi\vert^{2}_g\,g$ and its dual defines the canonical unit-norm parallel vector field that every manifold of type $S^1\times S^3$ carries. Furthermore, it can be seen that the restriction of $\vert\varphi\vert^{2}_g\,g$ to the kernel of $\vert\varphi\vert^{-2}_g\varphi$ precisely yields a metric of sectional curvature $\frac{1}{4}$ (see Definition \ref{def:Gauduchon}) by following the same steps as in the proof of Proposition \ref{prop:compactNSNS}. Hence, the assignment:
\begin{equation*}
(g,\varphi) \mapsto (\vert\varphi\vert^{2}_g\,g,\vert\varphi\vert^{-2}_g\varphi, \vert\varphi\vert_g)	\, ,
\end{equation*}  

\noindent
gives the desired bijection upon use of Theorem \ref{thm:bijection}.
\end{proof}

\begin{example}
	For $\Sigma = S^3$ we have $\mathrm{Iso}(S^3) = \mathrm{SO}(4)$ and the space of conjugacy classes $\cI(S^3) = T/W(S^3,T)$ admits a very explicit description. A maximal torus of $\mathrm{SO}(4)$ can be conjugated to a group of matrices of the form:
	
	\[
	\begin{bmatrix}
	\cos(x) & \sin(x) & 0 & 0 \\
	-\sin(x) & \cos(x) & 0 & 0 \\
	0 & 0 & \cos(y) & \sin(y) \\
	0 & 0 & -\sin(y)  & \cos(y)
	\end{bmatrix}
	\] 
	
	\
	
	\noindent
	where $x, y \in [0,2\pi]$. Hence, $T$ is a two torus and thus $\dim(\mathfrak{M}(S^3)) = 3$. Furthermore, the Weyl group can be shown to be the group of even signed permutations of two elements.
\end{example}


\subsection{Infinitesimal deformations of NS-NS pairs}
\label{sec:infdeformations}


We consider now the infinitesimal deformation problem of NS-NS structures on a manifold $M$ of type $S^1\times S^3$ around a fixed NS-NS pair $(g,\varphi)$ modulo the action of the diffeomorphism group of $M$, with the goal of obtaining the \emph{infinitesimal} counterpart of the results obtained in the previous Section. As we will see momentarily, the differential operator controlling the infinitesimal deformations of a given NS-NS pair has a nice geometric interpretation when restricted to an appropriate submanifold of $M$. Let $M$ be a compact four-manifold and let $\omega$ be a fixed volume form on $M$. We denote by $\Met_{\omega}(M)\subset\Gamma(T^{\ast}M^{\odot 2}) $ the space of Riemannian metrics on $M$ whose associated Riemannian volume form $\nu_g$ is equal to $\omega$. Using the equations defining the notion of NS-NS pair we introduce the following map:
\begin{eqnarray*}
	\cE = (\cE_1 , \cE_2, \cE_3, \cE_4)\colon \Met(M)\times \Omega^1(M) &\to & \Gamma(T^{\ast}M^{\odot 2})\times \Gamma(T^{\ast}M^{\odot 2}) \times \Omega^2(M)\times C^{\infty}(M)\, ,\\  
	(g,\varphi) &\mapsto & (\mathrm{Ric}^{g}+ \frac{1}{2} \varphi\otimes \varphi - \frac{1}{2} \vert \varphi \vert^2_g\, g , \mathcal{L}_{\varphi^{\sharp}}g , \dd\varphi, \vert\varphi\vert_g^2 - 1)\, ,
\end{eqnarray*}

\noindent
where $\mathcal{L}_{\varphi^{\sharp}}$ denotes Lie derivative along $\varphi^{\sharp}$, the metric dual of $\varphi$. Using the fact that $\nabla^g\varphi = 0$ if and only if $\mathcal{L}_{\varphi^{\sharp}}g=0$  and $\dd\varphi = 0$, it follows that the preimage  $\cE^{-1}(0)$ of $0$ by $\cE$ is by construction the set of all NS-NS pairs $(g,\varphi)$ on $M$ with unit norm $\varphi$ and inducing $\omega$ as Riemannian volume form of $g$. We assume that both $\Met_{\omega}(M)$ and $\Omega^1(M)$ are completed in the Sobolev norm $\mathrm{H}^s = \mathrm{L}^2_s$ with $s$ large enough so $\Met_{\omega}(M)\times \Omega^1(M)$ becomes a Hilbert manifold. The operator $\cE$ admits a canonical extension to the Sobolev completion of $\Met_{\omega}(M)\times \Omega^1(M)$, which we denote for ease of notation by the same symbol. The tangent space of $\Met_{\omega}(M)\times \Omega^1(M)$ at $(g,\varphi)$ is given by:
\begin{equation*}
T_{(g,\varphi)} (\Met_{\omega}(M)\times \Omega^1(M)) = \left\{ (\tau,\eta) \in \Gamma(T^{\ast}M^{\odot 2})\times\Omega^1(M)\,\, \vert\,\, \mathrm{Tr}_g(\tau) = 0\right\}\, ,
\end{equation*}

\noindent
which again is assumed to be completed in the appropriate Sobolev norm. The trace-less condition appearing in the previous equation occurs due to the fact that $\Met_{\omega}(M)$ is restricted to those Riemannian metrics inducing Riemannian volume forms equal to $\omega$. In the standard deformation problem of Einstein metrics such condition follows automatically simply from restricting to metrics of unit volume \cite{Besse}. For every Riemannian metric $g$ on $M$, we introduce the linear map of vector bundles:
\begin{equation*}
o^g\colon S^2 T^{\ast}M \to S^2 T^{\ast}M\, , \quad \tau\mapsto o^g(\tau)\, ,
\end{equation*}

\noindent
where, given a local orthonormal frame $\left\{ e_i \right\}$, we define:
\begin{equation*}
o^g(\tau)(v_1,v_2) = \sum_i \tau(\mR^g_{v_1,e_i}v_2,e_i)\, ,
\end{equation*}

\noindent
for every $v_1,v_2\in TM$. With this definition, the Lichnerowicz Laplacian restricted to symmetric $(2,0)$ tensors is given by \cite{Besse}:
\begin{equation*}
\Delta_L^g \tau = (\nabla^g)^{\ast}\nabla^g \tau + \mathrm{Ric}^g\circ_g \tau + \tau\circ_g \mathrm{Ric}^g - 2\,o^g(\tau)\, ,
\end{equation*} 

\noindent
where $(\nabla^g)^{\ast}$ is the adjoint of the Levi-Civita connection acting on $(2,0)$ tensors and the contraction $\circ_g$ is defined analogously to its counterpart for forms as introduced in Section \ref{sec:Heterotic4d}. In particular:
\begin{equation*}
(\mathrm{Ric}^g\circ_g \tau)(v_1,v_2) = g(\mathrm{Ric}^g(v_1), \tau(v_2))\, , \qquad v_1,v_2 \in \mathfrak{X}(M)\, ,
\end{equation*}

\noindent
and similarly for $\tau\circ_g \mathrm{Ric}^g$. Note that the 1-form $\varphi$ of a NS-NS pair $(g,\varphi)$ has constant norm, so for definiteness we will assume in the following that such $\varphi$ has in fact unit norm.

\begin{lemma}
\label{lemma:NSNSlinear}
Let $(g,\varphi)$ be a NS-NS pair. The differential of $\cE$ at $(g,\varphi)$ reads:
\begin{eqnarray*}
&\dd_{(g,\varphi)}\cE_1(\tau,\eta) = \frac{1}{2} \Delta^g_L (\tau) - 2\delta^{\ast}_g\delta_g \tau + \frac{1}{2}(\tau\otimes \varphi + \varphi \otimes \tau) - \frac{1}{2}\tau \, , \quad \dd_{(g,\varphi)}\cE_3(\tau,\eta) = \dd \eta\, ,\\
&\dd_{(g,\varphi)}\cE_2(\tau,\eta) = \mathcal{L}_{\eta^{\sharp}}g - \mathcal{L}_{(\varphi\lrcorner \tau)^{\sharp}}g + \mathcal{L}_{\varphi^\sharp} \tau\, ,  \quad \dd_{(g,\varphi)}\cE_4(\tau,\eta) = 2 g(\eta,\varphi) -  \tau(\varphi,\varphi)\, ,
\end{eqnarray*}

\noindent
where $\delta_g \tau$ denotes the divergence of $\tau$ and $\delta^{\ast}_g$ denotes the formal adjoint of $\delta_g$.
\end{lemma}

\begin{proof}
By definition, the (Gateaux) differential of the maps $\cE_a$, $a=1,\hdots 4$, at the point $(g,\varphi)$ and evaluated on $(\tau,\eta)\in T_{(g,\varphi)} (\Met_{\omega}(M)\times \Omega^1(M))$ is given by:
\begin{equation*}
\dd_{(g,\varphi)}\cE_a (\tau,\eta) = \lim_{t \to 0} \frac{\cE_a(g + t\, \tau, \varphi + t\, \eta) - \cE_a(g,\varphi)}{t}\, .
\end{equation*}

\noindent
On the other hand, recall that the differential of the map $(g,\varphi) \mapsto \varphi^{\sharp_g}$ at $(g,\varphi)$ along $(\tau,\eta)$ is given by $\eta^{\sharp_g} - (\varphi^{\sharp_g}\lrcorner \tau)^{\sharp_g}$. This immediately implies:
\begin{equation*}
\dd_{(g,\varphi)}\cE_4(\tau,\eta) = 2 g(\eta,\varphi) - \tau(\varphi,\varphi)\, .
\end{equation*}

\noindent
Furthermore, a direct computation, using that $\varphi$ has unit norm together with the previous equation, shows that:
\begin{equation*}
\dd_{(g,\varphi)}\cE_1 (\tau,\eta)  = \dd_{g}\mathrm{Ric} (\tau,\eta)  + \frac{1}{2}(\tau\otimes \varphi + \varphi \otimes \tau)  - \frac{1}{2} \tau  \, ,
\end{equation*}

\noindent
where $\dd_{g}\mathrm{Ric}$ denotes the differential of the Ricci map $\mathrm{Ric} \colon \Met_{\omega}(M) \to \Gamma(T^{\ast}M^{\odot 2})$. Computing this differential explicitly, see \cite[Equation (1.180a)]{Besse} gives the result in the statement upon use of $\mathrm{Tr}_g(\tau) = 0$. Similarly, computing for $\dd_{(g,\varphi)}\cE_2(\tau,\eta)$ we obtain:
\begin{equation*}
\dd_{(g,\varphi)}\cE_2(\tau,\eta) = \mathcal{L}_{\eta^{\sharp}}g - \mathcal{L}_{(\varphi^{\sharp}\lrcorner \tau)^{\sharp}}g+\mathcal{L}_{\varphi^\sharp} \tau\, ,
\end{equation*}

\noindent
where we have used, as remarked above, that the differential of the map $(g,\varphi) \mapsto \varphi^{\sharp_g}$ is given by $\eta^{\sharp} - (\varphi^{\sharp}\lrcorner \tau)^{\sharp}$. The differential of $\cE_3\colon \Met_{\omega}(M)\times \Omega^1(M) \to \Omega^2(M)$ follows easily since $\cE_3$ does not depend on $g$.
\end{proof}

\noindent
The kernel of $\dd_{(g,\varphi)}\cE\colon  T_{(g,\varphi)}(\Met_{\omega}(M)\times \Omega^1(M))  \to  \Gamma(T^{\ast}M^{\odot 2})\times \Gamma(T^{\ast}M^{\odot 2}) \times \Omega^2(M)\times C^{\infty}(M)$ describes the space of infinitesimal deformations of $(g,\varphi)$ that preserve the norm of $\varphi$ and the Riemannian volume form induced by $g$. These conditions eliminate the \emph{spurious} deformations given by constant rescalings of $\varphi$ or homotheties of the metric. The group of diffeomorphisms $\Diff_{\omega}(M)$ that preserves the fixed volume form $\omega$, again completed in the Sobolev norm $\mathrm{H}^s$, acts naturally on $\Met_{\omega}(M)\times \Omega^1(M)$ through pull-back. Recall that the tangent space of $\Diff_{\omega}(M)$ at the identity corresponds to the vector fields on $M$ that preserve $\omega$. This action preserves $\cE^{-1}(0)$ and hence maps solutions to solutions. The moduli space of NS-NS pairs $(g,\varphi)$ with constant norm $\varphi$ and associated Riemannian volume form equal to $\omega$ is defined as:
\begin{equation*}
\mathfrak{M}^0_{\omega}(M) := \cE^{-1}(0)/\Diff_{\omega}(M)\, ,
\end{equation*}

\noindent
endowed with the quotient topology. Define:
\begin{equation*}
\cO_{(g,\varphi)} := \left\{ (u^{\ast}g,u^{\ast}\varphi) \,\, \vert\,\, u\in \Diff_{\omega}(M)\right\}\, ,
\end{equation*}

\noindent
to be the orbit of the diffeomorphism group passing through $(g,\varphi)$. The tangent space to the orbit at $(g,\varphi)\in \cO_{(g,\varphi)}$ can be computed to be:
\begin{equation*}
T_{(g,\varphi)}\cO_{(g,\varphi)} = \left\{ (\mathcal{L}_v g, \dd(\iota_v\varphi))\, , \,\, v\in \mathfrak{X}(M) \,\, \vert \,\, \mathcal{L}_v \omega = 0\right\}\, ,
\end{equation*}

\noindent
where $\mathcal{L}$ denotes the Lie derivative. 

\begin{lemma}
The vector subspace  $T_{(g,\varphi)}\cO_{(g,\varphi)}\subset T_{(g,\varphi)} (\Met_{\omega}(M)\times \Omega^1(M))$ is closed in the Hilbert space $T_{(g,\varphi)} (\Met_{\omega}(M)\times \Omega^1(M))$.
\end{lemma}

\begin{proof}
Follows from the fact that the differential operator $\mathfrak{X}(M)\ni v\mapsto (\mathcal{L}_v g, \dd\iota_v\varphi)$ has injective symbol.
\end{proof}

\noindent
By the previous lemma, the $L^2$ orthogonal complement of $T_{(g,\varphi)}\cO_{(g,\varphi)}$ is a Hilbert subspace of $T_{(g,\varphi)} (\Met_{\omega}(M)\times \Omega^1(M))$, which is given by:
\begin{equation*}
T_{(g,\varphi)}\cO_{(g,\varphi)}^{\perp} = \left\{ (\tau,\eta)\in T_{(g,\varphi)}\cO_{(g,\varphi)}  \,\, \vert \,\,  (\nabla^g)^{\ast}\tau = 0\, , \,\, \delta^g \eta = 0 \right\}\, .
\end{equation*}

\noindent
By an extension of the celebrated Ebin's slice theorem \cite{Ebin}, for every pair $(g,\varphi)$ the action of $\Diff_{\omega}(M)$ on $\Met_{\omega}(M)\times \Omega^1(M)$ admits a \emph{slice} $\cS_{(g,\varphi)}$ whose tangent space at $(g,\varphi)$ is precisely $T_{(g,\varphi)}\cO_{(g,\varphi)}^{\perp}$. Therefore, by applying standard Kuranishi theory for differential-geometric moduli spaces, the \emph{virtual} tangent space of $\mathfrak{M}^0_{\omega}(M)$ at the equivalence class $[g,\varphi]$ defined by $(g,\varphi)\in \cE^{-1}(0)$ in $\mathfrak{M}^0_{\omega}(M)$, is given by:
\begin{equation*}
T_{[g,\varphi]}\mathfrak{M}^0_{\omega}(M) := \Ker(\dd_{(g,\varphi)}\cE) \cap \Ker((\nabla^g)^{\ast}\oplus \delta^g)\, .
\end{equation*}

\noindent
Using the terminology introduced by Koiso \cite{Koiso,KoisoII} in the study of deformations of Einstein metrics and Yang-Mills connections, we will call elements of $T_{[g,\varphi]}\mathfrak{M}^0_{\omega}(M)$ \emph{essential deformations} of $(g,\varphi)$. Roughly speaking, essential deformations are infinitesimal deformations of $(g,\varphi)$ that cannot be eliminated via the infinitesimal action of the diffeomorphism group.

\begin{lemma}
\label{lemma:linearNSNSII}
The pair $(\tau,\eta)\in \Gamma(T^{\ast}M^{\odot 2})\times \Omega^1(M)$ is an essential deformation of the NS-NS pair $(g,\varphi)$ if and only if $\eta = \lambda\, \varphi$ for a constant $\lambda \in \mathbb{R}$ and the following equations are satisfied:
\begin{eqnarray}
\label{eq:linearNSNSII}
\Delta^g_L\tau  +  2 \lambda\,\varphi\otimes \varphi -  \tau = 0\, ,\quad \mathcal{L}_{(\varphi^{\sharp}\lrcorner \tau)^{\sharp}}g  = \mathcal{L}_{\varphi^\sharp} \tau\, , \quad (\nabla^g)^{\ast}\tau = 0\, , \quad \mathrm{Tr}_g(\tau) = 0\, .
\end{eqnarray}
\end{lemma}
 
\begin{proof}
A pair $(\tau,\eta) \in \Gamma(T^{\ast}M^{\odot 2})\times \Omega^1(M)$ is an essential deformation if and only if:
\begin{equation*}
\dd_{(g,\varphi)}\cE(\tau,\eta) = 0\, , \quad (\nabla^g)^{\ast}\tau = 0\, , \qquad \delta^g\eta =0\, , \quad \mathrm{Tr}_g(\tau) = 0\, .
\end{equation*}

\noindent
By Lemma \ref{lemma:NSNSlinear}, we have $\dd_{(g,\varphi)}\cE_3(\tau,\eta) = \dd \eta$ hence if $(\tau,\eta)$ is an essential deformation then $\eta$ is closed and co-closed whence harmonic. Since $b^1(M) = 1$ and $\varphi$ is parallel, in particular harmonic, we conclude that $\eta = \lambda \varphi$ for a real constant $\lambda \in \mathbb{R}$. Plugging $\eta = \lambda \varphi$ into the explicit expression of $\dd_{(g,\varphi)}\cE(\tau,\eta) = 0$, given in Lemma \ref{lemma:NSNSlinear}, we obtain equations \eqref{eq:linearNSNSII} and hence we conclude.
\end{proof}

\noindent
Since, by assumption, $M$ admits NS-NS pairs $(g,\varphi)$ with non-vanishing Lee class $[\varphi]$, Proposition \ref{prop:compactNSNS} implies that $(M,g)$ is a manifold of type $S^1\times S^3$ and, consequently, it is a fibre bundle over $S^1$ with fiber $\Sigma = S^3/\Gamma$, $\Gamma\subset \mathrm{SO}(4)$, as described in Section \ref{sec:NS-NSmoduli}. For simplicity in the exposition, we will assume that $(M,g)$ is isometric to:
\begin{equation*}
(M,g) = (S^1\times \Sigma, \varphi\otimes \varphi + h)\, ,
\end{equation*}
  
\noindent
where $h$ is a Riemannian metric on $\Sigma$. Analogous results can be obtained in the general case by using the integrable distribution defined by the kernel of $\varphi$. Given $\tau\in \Gamma(T^{\ast}M^{\odot 2})$ we decompose it according to the orthogonal decomposition defined by $g$, that is:
\begin{equation*}
\tau = \mathfrak{f}\, \varphi\otimes \varphi + \varphi \odot \beta + \tau^{\perp}\, ,
\end{equation*}

\noindent
where the superscript $\perp$ denotes projection along $\Sigma$ and $\beta$ is a 1-form along $\Sigma$.

\begin{proposition}
\label{prop:iffessential}
The pair $(\tau = \mathfrak{f}\, \varphi\otimes \varphi + \varphi \odot \beta + \tau^{\perp},\eta = \lambda\, \varphi)\in \Gamma(T^{\ast}M^{\odot 2})\times \Omega^1(M)$ is an essential deformation of the NS-NS pair $(g,\varphi)$ only if:
\begin{eqnarray*}
\lambda =0 \, , \qquad \mathfrak{f}=0 \, , \qquad \nabla^g_{\varphi^{\sharp}} \beta = 0\, , \qquad \tau^{\perp} = 0\, , \qquad \mathcal{L}_{\beta^{\sharp}} h = 0\, .
\end{eqnarray*}
\end{proposition}
 
\begin{proof}
A pair $(\tau,\eta)$ is an essential deformation if and only if conditions \eqref{eq:linearNSNSII} hold. Given the decomposition $\tau = \mathfrak{f}\, \varphi\otimes \varphi + \varphi \odot \beta + \tau^{\perp}$, we impose first the \emph{slice} condition $(\nabla^g)^{\ast}\tau = 0$. We obtain:
\begin{equation*}
(\nabla^g)^{\ast}\tau = - \dd\mathfrak{f}(\varphi^{\sharp}) \varphi -  \nabla^g_{\varphi^{\sharp}} \beta + \varphi \, \delta^g\beta  +  (\nabla^h)^{\ast}\tau^{\perp} = 0\, , 
\end{equation*}

\noindent
hence $\dd\mathfrak{f}(\varphi^{\sharp}) = \delta^g\beta$ and $\nabla^g_{\varphi} \beta = (\nabla^h)^{\ast}\tau^{\perp}$. On the other hand, equation $\mathcal{L}_{(\varphi^{\sharp}\lrcorner \tau)^{\sharp}}g =  \mathcal{L}_{\varphi^\sharp} \tau$ reduces to:
\begin{equation*}
\dd\mathfrak{f} = 0\, , \qquad \varphi \odot \nabla^g_{\varphi^{\sharp}} \beta +\mathcal{L}_{\varphi^\sharp} \tau^\perp=\mathcal{L}_{\beta^\sharp} h\, ,
\end{equation*}

\noindent
where we have used that $\varphi^{\sharp}\lrcorner \tau = \mathfrak{f}\, \varphi + \beta$. Hence, isolating by type we obtain $ \nabla^g_{\varphi^{\sharp}} \beta = 0$ and $\mathcal{L}_{\varphi^\sharp} \tau^\perp=\mathcal{L}_{\beta^\sharp} h$. Note that since $\mathfrak{f}$ is constant we have  $\delta^g\beta = 0$. We decompose now the first equation in \eqref{eq:linearNSNSII}. For this, we first compute:
\begin{equation*}
\mathrm{Ric}^g\circ_g \tau + \tau\circ_g \mathrm{Ric}^g = \frac{1}{2} (h\circ \tau + \tau \circ h) = \frac{1}{2} \varphi \odot \beta + \tau^{\perp}\, ,
\end{equation*}

\noindent
as well as:
\begin{equation*}
(\nabla^g)^{\ast} \nabla^g\tau = (\nabla^g)^{\ast} \nabla^g (\mathfrak{f}\, \varphi\otimes \varphi + \varphi \odot \beta + \tau^{\perp}) = \varphi \odot (\nabla^g)^{\ast} \nabla^g\beta + (\nabla^g)^{\ast} \nabla^g\tau^{\perp}\, ,
\end{equation*}

\noindent
which in turn implies:
\begin{eqnarray*}
&\Delta_L^g \tau = \varphi \odot (\nabla^g)^{\ast} \nabla^g\beta + (\nabla^g)^{\ast} \nabla^g\tau^{\perp} + \frac{1}{2} \varphi \odot \beta + \tau^{\perp} - 2 o^h(\tau^{\perp})\\
& = \Delta_L^h \tau^{\perp} + \varphi \odot (\frac{1}{2} \beta + (\nabla^g)^{\ast} \nabla^g\beta) \, .
\end{eqnarray*}

\noindent
Hence, the first equation in \eqref{eq:linearNSNSII} is equivalent to:
\begin{equation*}
\Delta_L^h \tau^{\perp} + \varphi \odot ((\nabla^g)^{\ast} \nabla^g\beta - \frac{1}{2} \beta) + 2 \lambda\,\varphi\otimes \varphi -  \mathfrak{f}\, \varphi\otimes \varphi  - \tau^{\perp} = 0\, .
\end{equation*}

\noindent
Solving by type, we obtain:
\begin{equation*}
\Delta_L^h \tau^{\perp} = \tau^{\perp}\, , \quad (\nabla^g)^{\ast} \nabla^g\beta = \frac{1}{2} \beta\, , \quad \mathfrak{f} = 2\lambda\, .
\end{equation*}

\noindent
Solutions to the first equation above correspond to infinitesimal essential Einstein deformations of $(\Sigma,h)$, which by \cite{KoisoIII} are necessarily trivial since $(\Sigma,h)$ is covered by the round sphere. Hence $\tau^{\perp} = 0$. This in turn implies $\mathcal{L}_{\beta^\sharp} h=0$. The second equation above follows automatically from $\beta^{\sharp}$ being a Killing vector field on an Einstein three-manifold with Einstein constant $1/2$. Moreover, the third equation above uniquely determines $\mathfrak{f}$ in terms of $\lambda$. Putting all together, we obtain:
\begin{equation*}
\tau = 2\lambda \, \varphi\otimes \varphi + \varphi \odot \beta\, .
\end{equation*} 

\noindent
With these provisos in mind, equation $\mathrm{Tr}_g(\tau) = 0$ is equivalent to $\lambda = 0$ whence:
\begin{equation*}
\tau = \varphi \odot \beta\, .
\end{equation*} 

\noindent
Conversely, such $\tau$ solves all equations in \eqref{eq:linearNSNSII} with $\eta = 0$ and hence we conclude.
\end{proof}

\noindent
The previous proposition shows that the 1-form $\beta$ descends to a 1-form on $\Sigma$ whose metric dual is a Killing vector field of $h$. Denote by $\cK(\Sigma,h)$ the vector space of Killing vector fields on $(\Sigma,h)$.

\begin{theorem}
\label{thm:infbijection}
There exists a canonical bijection:
\begin{equation*}
T_{[g,\varphi]}\mathfrak{M}^0_{\omega}(M) \to   \cK(\Sigma,h)\, , \quad (\tau,0) \mapsto \beta^{\sharp} \, ,
\end{equation*}

\noindent
where, for every $(\tau,0)\in T_{[g,\varphi]}\mathfrak{M}(M)$ we write uniquely $\tau = \varphi \odot \beta$.
\end{theorem}

\begin{proof}
By Lemma \ref{lemma:linearNSNSII} and Proposition \ref{prop:iffessential} a pair $(\tau,\eta)$ is an essential deformation, that is, belongs to $T_{[g,\varphi]}\mathfrak{M}^0_{\omega}(M)$ if and only if $\eta = 0$ and $\tau =  \varphi \otimes \beta$
for a Killing vector field $\beta^{\sharp}$. This implies the statement of the theorem.
\end{proof}

\noindent
Taking $(\Sigma,h)$ to be the round sphere and assuming $M= S^1\times S^3$ we have $\dim(\cK(\Sigma,h)) = 6$ and thus $\dim(T_{[g,\varphi]}\mathfrak{M}^0_{\omega}(S^1\times S^3)) = 6$. On the other hand, in Section \ref{sec:ModuliS1S3} we constructed the full moduli space of manifolds of type $S^1\times S^3$ and in the case in which $(\Sigma,h)$ is the round sphere we proved that it was two-dimensional after removing the spurious deformation consisting in rescalings of $\varphi$. Since  $\dim(T_{[g,\varphi]}\mathfrak{M}^0_{\omega}(S^1\times S^3)) = 6$, we conclude that the space of essential deformations is obstructed and there exist \emph{four directions} in $T_{[g,\varphi]}\mathfrak{M}^0_{\omega}(S^1\times S^3)$ which cannot be integrated and therefore do not correspond to honest deformations. 
 

\section{Heterotic solitons with parallel torsion}
\label{sec:constantdilatonsection}

 
In this section we restrict our attention to Heterotic solitons with constant dilaton and parallel non-vanishing torsion, that is, Heterotic solitons that satisfy $\varphi = 0$ and $\nabla^g\alpha = 0$ with $\alpha \neq 0$. These Heterotic solitons with constant dilaton, in the specific case of four dimensions, can never be supersymmetric since the second equation in \eqref{eq:susytransHeterotic} is equivalent to $\alpha = \varphi$. Therefore, this class of Heterotic solitons provides a convenient framework to explore non-supersymmetric solutions of Heterotic supergravity.  


\subsection{Null Heterotic solitons}


Assuming $\varphi = 0$, the Heterotic soliton system \eqref{eq:HeteroticRicci1}--\eqref{eq:HeteroticRicci2} reduces to the following system of equations:
\begin{eqnarray}
\label{eq:HeteroticRicci11}
&\mathrm{Ric}^{g} +  \frac{1}{2} \alpha\otimes \alpha - \frac{1}{2} \vert\alpha\vert^2_g \, g + \kappa\,\mathfrak{v}(\mR_{\nabla^{\alpha}} \circ\mR_{\nabla^{\alpha}}) = 0\, , \\ 
\label{eq:HeteroticRicci22}
& \dd\alpha = 0\, , \quad \delta^g\alpha =\kappa (|\mR^{+}_{\nabla^{\alpha}}|^2_{g,\mathfrak{v}} - |\mR^{-}_{\nabla^{\alpha}}|^2_{g,\mathfrak{v}})\, , \quad  \kappa |\mR_{\nabla^{\alpha}}|^2_{g,\mathfrak{v}}   =  \vert\alpha\vert^2_g  \, ,  
\end{eqnarray}

\noindent
for pairs $(g,\alpha)$, where $g$ is a Riemannian metric on $M$ and $\alpha\in \Omega^1(M)$ is a 1-form. 

\begin{definition}
The \emph{null Heterotic soliton system} consists of equations \eqref{eq:HeteroticRicci11} and \eqref{eq:HeteroticRicci22}. Solutions of the null Heterotic soliton system are \emph{null Heterotic solitons}.
\end{definition}

\noindent
In the following we will study a particular case of the null Heterotic soliton system that is obtained by imposing $\alpha$ to be parallel. Assuming that $\nabla^g\alpha = 0$, the null Heterotic soliton system further reduces to:
\begin{eqnarray}
	\label{eq:motionHetconstantvarphi11}
	&\mathrm{Ric}^{g} +  \frac{1}{2} \alpha\otimes \alpha - \frac{1}{2} \vert\alpha\vert^2_g \, g + \kappa\,\mathfrak{v}(\mR_{\nabla^{\alpha}} \circ\mR_{\nabla^{\alpha}}) = 0\, , \\ 
	\label{eq:motionHetconstantvarphi22}
	&  |\mR^{+}_{\nabla^{\alpha}}|^2_{g,\mathfrak{v}} = |\mR^{-}_{\nabla^{\alpha}}|^2_{g,\mathfrak{v}}\, , \quad  \kappa |\mR_{\nabla^{\alpha}}|^2_{g,\mathfrak{v}}   =  \vert\alpha\vert^2_g  \, ,  
\end{eqnarray}

\noindent
Throughout this section, $\Conf_{\kappa}(M)$ will denote the set of pairs $(g,\alpha)$ as described above, with $\alpha$ being a non-vanishing 1-form satisfying $\nabla^g\alpha = 0$, and $\Sol_{\kappa}(M)$ will denote the space of null Heterotic solitons $(g,\alpha)$  with parallel 1-form $\alpha$. Also, we shall denote a vector and its metric dual by the same symbol. A direct computation proves the following lemma.

\begin{lemma}
\label{lemma:formulasparallel}
Let $\alpha$ be a parallel $1$-form. The following formulas hold:
\begin{eqnarray*}
& \mR^{\nabla^{\alpha}}_{v_1,v_2} = \mR^g_{v_1, v_2} + \frac{1}{4}(\vert\alpha\vert^2_g v_1\wedge v_2 + \alpha(v_2) \alpha\wedge v_1 - \alpha(v_1) \alpha\wedge v_2)\in \Omega^2(M)\, , \quad \forall\,\, v_1 , v_2 \in TM\, ,\\
& \mathfrak{v}(\mR^{\nabla^{\alpha}}\circ \mR^{\nabla^{\alpha}}) = \mathfrak{v}(\mR^{g}\circ \mR^{
	g}) -\vert\alpha\vert^2_g \mathrm{Ric}^g + \frac{\vert\alpha\vert^2_g}{4} (\vert\alpha\vert^2_g g - \alpha\otimes \alpha)\, ,
\end{eqnarray*} 
	
\noindent
where the curvature tensor is defined by $ \mR^{\nabla^{\alpha}}_{v_1,v_2}: = \nabla^{\alpha}_{v_1}\nabla^{\alpha}_{v_2} - \nabla^{\alpha}_{v_2}\nabla^{\alpha}_{v_1} -\nabla^{\alpha}_{[v_1,v_2]}$.
\end{lemma}

\noindent
Exploiting the fact that $\alpha$ is parallel, equations \eqref{eq:motionHetconstantvarphi11} and \eqref{eq:motionHetconstantvarphi22} can be further simplified.

\begin{lemma}
Let $(g,\alpha)\in \Conf_{\kappa}(M)$. Then:
\begin{equation*}
|\mR^{+}_{\nabla^{\alpha}}|^2_{g,\mathfrak{v}} = |\mR^{-}_{\nabla^{\alpha}}|^2_{g,\mathfrak{v}}	\, ,
\end{equation*}

\noindent
whence the first equation in \eqref{eq:motionHetconstantvarphi22} automatically holds for every $(g,\alpha)\in \Conf_{\kappa}(M)$.
\end{lemma}

\begin{proof}
The equality $|\mR^{+}_{\nabla^{\alpha}}|^2_{g,\mathfrak{v}} = |\mR^{-}_{\nabla^{\alpha}}|^2_{g,\mathfrak{v}}$ holds if and only if:
\begin{equation*}
\langle \mR_{\nabla^{\alpha}}, \ast \mR_{\nabla^{\alpha}} \rangle_g = 0\, .
\end{equation*}

\noindent
The fact that $\alpha$ is parallel implies $\alpha \lrcorner \mR_{\nabla^{\alpha}} = 0$.  Consequently, one can write:
\begin{equation*}
\ast \mR_{\nabla^{\alpha}} = \alpha\wedge r\, ,
\end{equation*}

\noindent
for a certain $\mathfrak{so}_g(M)$-valued 1-form $r\in \Omega^1(M,\mathfrak{so}_g(M))$. Therefore:
\begin{equation*}
\langle \mR_{\nabla^{\alpha}}, \ast \mR_{\nabla^{\alpha}} \rangle_g = \langle \mR_{\nabla^{\alpha}}, \alpha\wedge r\rangle_g = \langle \alpha \lrcorner \mR_{\nabla^{\alpha}}, r\rangle_g = 0\, ,
\end{equation*}

\noindent
and we conclude.
\end{proof}

\begin{remark}
\label{remark:formulacurvature3d}
In the following we will use on several occasions the following identity:
\begin{equation*}
\mR^h_{v_1,v_2} = \frac{s^h}{2} v_1\wedge v_2 + v_2\wedge \mathrm{Ric}^h(v_1) + \mathrm{Ric}^h(v_2)\wedge v_1\, , \qquad v_1,v_2 \in TN\, ,
\end{equation*}

\noindent
which yields the Riemann curvature tensor of a Riemannian metric $h$ on a three-dimensional manifold $N$ in terms of its Ricci curvature $\mathrm{Ric}^h$ and its scalar curvature $s^h$. In particular, using the previous formula it is easy to show that the contraction $\frv(\mR^h\circ \mR^h)$, defined exactly as we did in four dimensions in Section \ref{sec:Heterotic4d}, is given by:
\begin{equation}
\label{eq:RhoRh}
\frv(\mR^h\circ \mR^h) =  - 2\, \mathrm{Ric}^h\circ \mathrm{Ric}^h + 2 s^h \mathrm{Ric}^h + (2\vert \mathrm{Ric}^h\vert^2_h - (s^h)^2) h\, .
\end{equation}

\noindent
where:
\begin{equation*}
\mathrm{Ric}^h\circ \mathrm{Ric}^h(v_1,v_2) = h(\mathrm{Ric}^h(v_1),\mathrm{Ric}^h(v_2))\, , \qquad v_1, v_2\in TN\, .
\end{equation*}

\noindent
In particular, the norm of $\mathrm{R}^h$ is given by:
\begin{equation*}
\vert \mathrm{R}^h \vert_h^2 = \frac{1}{2} \mathrm{Tr}_h(\frv(\mR^h\circ \mR^h))  =  2 \vert \mathrm{Ric}^h \vert_h^2 - \frac{1}{2} (s^h)^2	\, .
\end{equation*}
\end{remark}

\noindent
Given a pair $(g,\alpha)\in \Conf_{\kappa}(M)$ we denote by $\cH\subset TM$ the rank-three distribution defined by the kernel of $\alpha$, which is integrable since the latter is parallel. We denote the corresponding foliation by $\cF_{\alpha}\subset M$.  

\begin{lemma}
\label{lemma:equationsparallel}
Let $(g,\alpha)\in \Conf_{\kappa}(M)$ be complete. Then, $(g,\alpha)\in \Sol_{\kappa}(M)$  if and only if the leaves of $\cF_{\alpha}$ endowed with the metric induced by $g$ are all isometric to a complete Riemannian three-manifold $(\Sigma,h)$ satisfying:
\begin{eqnarray}
\label{eq:equationsparallel1}
& -2\kappa\, \mathrm{Ric}^h\circ \mathrm{Ric}^h + (1 - 2\kappa \vert\alpha\vert_g^2) \mathrm{Ric}^h +\frac{\vert\alpha\vert_g^2}{2}(1 - \kappa\,\vert\alpha\vert_g^2 ) h = 0\, , \quad s^h = -\frac{1}{2} \vert\alpha\vert_g^2\, , 
\end{eqnarray}
	
\noindent
for a certain $\kappa > 0$. In particular, $\vert\mathrm{Ric}^h\vert^2_h = \frac{\vert\alpha\vert_g^2}{2\kappa}(1 -\frac{\kappa \vert\alpha\vert_g^2}{2} )$.
\end{lemma}

\begin{proof}
If $g$ is complete then standard results in foliation theory imply that $\cF_{\alpha}$ has no holonomy and its leaves are all diffeomorphic. Furthermore, since $\alpha$ is parallel it is in particular Killing and its flow preserves the metric, whence all leaves are not only diffeomorphic but isometric to a Riemannian three-manifold $(\Sigma,h)$ when equipped with the metric induced by $g$. Using the fact that Equation \eqref{eq:motionHetconstantvarphi11} evaluated in $\alpha$ is automatically satisfied, it follows that it is equivalent to its restriction to $\cH \otimes \cH$. Since all the leaves are isometric, Equation \eqref{eq:motionHetconstantvarphi11} is satisfied if and only if its restriction to any leaf is satisfied. Denoting this leaf by $(\Sigma,h)$, where $h$ is the metric induced by $g$, the restriction of Equation \eqref{eq:motionHetconstantvarphi11} to $\Sigma$ reads:
\begin{eqnarray*}
(1-\kappa \vert\alpha\vert^2_g )\mathrm{Ric}^{h}   - \frac{1}{2} \vert\alpha\vert^2_g\, h + \kappa\,(\mathfrak{v}(\mR^{h}\circ \mR^{h})  + \frac{\vert\alpha\vert^4_g}{4} h )= 0\, .
\end{eqnarray*}

\noindent
where we have used Lemma \ref{lemma:formulasparallel} to expand $\mathfrak{v}(\mR^{h}\circ \mR^{h})$. Plugging now Equation \eqref{eq:RhoRh} into the previous equation, we obtain:
\begin{eqnarray}
\label{eq:prooflemmaparallelEinstein}
-2\kappa\, \mathrm{Ric}^{h}\circ \mathrm{Ric}^{h} + 	(1-\kappa \vert\alpha\vert^2_g + 2 \kappa s^h )\mathrm{Ric}^{h}  + \left ( \kappa \frac{\vert\alpha\vert^4_g}{4}- \frac{1}{2} \vert\alpha\vert^2_g + 2\kappa \vert \mathrm{Ric}^h\vert_h^2 - \kappa (s^h)^2 \right ) h = 0\, .
\end{eqnarray}

\noindent
Moreover, using Equation \eqref{eq:RhoRh} and Lemma \ref{lemma:formulasparallel} it can be seen that the second equation in \eqref{eq:motionHetconstantvarphi22}, $\kappa \vert \mR_{\nabla^{\alpha}} \vert_{g,\frv}^2 = \vert\alpha\vert^2_g$, is equivalent to:
\begin{equation*}
4\vert\mathrm{Ric}^h\vert^2_h - (s^h)^2  - \vert\alpha\vert^2_g s^h + \frac{3}{4} \vert\alpha\vert^4_g = \frac{2}{\kappa}\vert\alpha\vert^2_g\, .
\end{equation*}

\noindent
Combining this equation together with the trace of the previous equation we can isolate both $\vert\mathrm{Ric}^h\vert^2_h$ and $s^h$. Upon substitution into Equation \eqref{eq:prooflemmaparallelEinstein}, we get Equations \eqref{eq:equationsparallel1}.
\end{proof}

\begin{proposition}
\label{prop:muconditions}
Let $(g,\alpha)\in \Conf_{\kappa}(M)$ be complete and non-flat. Then, $(g,\alpha)\in \Sol_{\kappa}(M)$ is a null Heterotic soliton with parallel torsion if and only if  $2 \kappa \vert\alpha\vert^2_g\in\{1,2,3\}$ and 
 the leaves of $\cF_{\alpha}$ endowed with the metric induced by $g$ are all isometric to a complete Riemannian three-manifold $(\Sigma,h)$ whose principal Ricci curvatures $(\mu_1,\mu_1,\mu_2)$ are constant and satisfy:
\begin{itemize}[leftmargin=*]
	\item $\mu_1 = -\frac{1}{4\kappa}\, ,  \quad \mu_2 =  \frac{1}{4\kappa}$ if $2 \kappa \vert\alpha\vert^2_g = 1$.
	
	\

	\item $\mu_1 = 0\, ,  \quad \mu_2 =  -\frac{1}{2\kappa}$ if $2 \kappa \vert\alpha\vert^2_g = 2$.
	
	\
	
	\item $\mu_1 = -\frac{1}{4\kappa}\, ,  \quad \mu_2 =-\frac{1}{4\kappa}$ if $2 \kappa \vert\alpha\vert^2_g =3$.
\end{itemize}

\end{proposition}
 
\begin{proof}
By Lemma \ref{lemma:equationsparallel}, a pair $(g,\alpha)\in \Conf_{\kappa}(M)$ is a solution of Equations \eqref{eq:motionHetconstantvarphi11} and \eqref{eq:motionHetconstantvarphi22} if and only if Equations \eqref{eq:equationsparallel1} are satisfied. The first equation in \eqref{eq:equationsparallel1} gives a second-degree polynomial satisfied by the Ricci endomorphism of $h$, whose roots are $-\frac{\vert\alpha\vert^2_g}{2}$ and $\frac{1-\vert\alpha\vert^2_g \, \kappa}{2 \kappa}$. Therefore, solving the algebraic equation we find that the principal Ricci curvatures $(\mu_1,\mu_1,\mu_2)$ of $h$ are constant and given by one of the following possibilities:
\begin{eqnarray*}
& (\mu_1 = -\frac{\vert\alpha\vert^2_g}{2} , \mu_2 = -\frac{\vert\alpha\vert^2_g}{2} )\, , \quad (\mu_1 = -\frac{\vert\alpha\vert^2_g}{2} , \mu_2 = \frac{1-\vert\alpha\vert^2_g \, \kappa}{2 \kappa} )\, , \\
& (\mu_1 =   \frac{1-\vert\alpha\vert^2_g \, \kappa}{2 \kappa} , \mu_2 =  -\frac{\vert\alpha\vert^2_g}{2} )\, , \quad (\mu_1 =   \frac{1-\vert\alpha\vert^2_g \, \kappa}{2 \kappa} , \mu_2 = \frac{1-\vert\alpha\vert^2_g \, \kappa}{2 \kappa}  )\, ,
\end{eqnarray*}

\noindent
Imposing now that the scalar curvature of $h$ is $s^h = 2\mu_1+\mu_2$ and using the second equation in \eqref{eq:equationsparallel1}, we obtain the cases and relations given in the statement of the proposition.
\end{proof}

\begin{remark}
\label{remark:universalcover}
Since $\alpha$ is by assumption parallel, if $(g,\alpha)\in \Sol_{\kappa}(M)$ is complete then the lift $(\hat{g},\hat{\alpha})$ of $(g,\alpha)$ to the universal cover $\hat{M}$ of $M$ of is isometric to the following model:
\begin{equation*}
(\hat{M},\hat{g},\hat{\alpha}) = (\mathbb{R}\times N, \dd t^2 + \hat{h} \, , \vert \alpha\vert_g  \dd t)\, ,
\end{equation*}
	
\noindent
where $N$ is a simply connected three-manifold, $t$ is the Cartesian coordinate of $\mathbb{R}$ and $\hat{h}$ is a complete Riemannian metric on $N$ whose principal Ricci curvatures satisfy the conditions established in Proposition \ref{prop:muconditions}. Moreover, the foliation $\cF_{\alpha}\subset M$ associated to $\alpha$ is induced by the standard foliation of $\mathbb{R}\times N$ whose leaves are given by $\left\{ x\right\} \times N \subset \mathbb{R}\times N$, $x\in \mathbb{R}$.
\end{remark}

\noindent
As stated in Proposition \ref{prop:muconditions}, the principal Ricci curvatures are constant and they can take at most two different values, $\mu_1$ and $\mu_2$. Suppose $\mu_1 \neq \mu_2$ and assume that $\mu_2$ is the eigenvalue of simple multiplicity. The eigenvectors with eigenvalue $\mu_2$ define a rank-one distribution  $\cV\subset T\Sigma$, which may not be trivializable. Therefore, going perhaps to a covering of $\Sigma$, we assume that $\cV$ is trivializable and fix a unit trivialization $\xi\in \Gamma(\cV)$, whose metric dual we denote by $\eta\in \Omega^1(\Sigma)$. We define the endomorphism $\cC\in \End(T\Sigma)$ as follows:
\begin{equation*}
\cC(v) := \nabla^h_v\xi\, , \quad v\in T\Sigma\, ,
\end{equation*}

\noindent
which we split $\cC = \cA + \cS$ in its antisymmetric $\cA$ and symmetric $\cS$ parts.

\begin{lemma}
\label{lemma:equationsC}
Assume $\mu_1 \neq \mu_2$. The following formulas hold:  
\begin{eqnarray*}
& \nabla^h_{\xi} \xi = 0\, , \quad \delta^h \eta = 0\, , \quad \mathrm{Tr}(\cC) = 0\, ,\quad \nabla^h_{\xi} \cC = 0\, , \quad \cC^2 = - \frac{\mu_2}{2} \Id_{\cH}\, ,\\
& \mathcal{L}_{\xi} \cC = 0\, , \quad \mathcal{L}_{\xi} \cA = -2 \cS\cA\, , \quad \mathcal{L}_{\xi} \cS = 2\cS\cA\, , \quad \mathcal{L}_{\xi} \dd\eta = 0\, , 
\end{eqnarray*}

\noindent
where $\cH$ is the orthogonal complement of $\cV$ in $T\Sigma$.
\end{lemma}

\begin{proof}
The condition of $h$ having a constant Ricci eigenvalue $\mu_1$ of multiplicity two and a simple constant Ricci eigenvalue $\mu_2$ is equivalent to $h$ satisfying:
\begin{equation*}
\mathrm{Ric}^h = \mu_1\,h + (\mu_2 - \mu_1)\, \eta\otimes \eta\, , \quad s^h = 2\mu_1 + \mu_2\, ,
\end{equation*}

\noindent
where $\eta$ is the metric dual of a unit eigenvector with eigenvalue $\mu_2$. Since $\mu_1\neq \mu_2$, the divergence of the previous equation together with the contracted Bianchi identity yields:
\begin{equation*}
\nabla^h_{\xi}\eta = \eta \,\delta^h\eta\, ,
\end{equation*}

\noindent
which in turn implies, using that $\xi$ is of unit norm, $\nabla^h_{\xi}\xi = 0$, $\delta^h \eta = 0$ and consequently $\mathrm{Tr}(\cC) = 0$. Furthermore, for every vector field $v$ orthogonal to $\xi$ we compute:
\begin{equation*}
\dd \eta(\xi,v) = - \eta(\nabla^h_{\xi} v - \nabla^h_{v} \xi) = -\eta(\nabla^h_{\xi} v) = - h(\xi, \nabla^h_{\xi} v) = 0\, ,
\end{equation*}

\noindent
implying $\mathcal{L}_{\xi} \dd\eta = 0$. Since $\cC$ is trace-free, its square satisfies:
\begin{equation*}
\cC^2 = \frac{1}{2} \mathrm{Tr}(\cC^2) \Id_{\cH}\, .
\end{equation*} 

\noindent
On the other hand, using Remark \ref{remark:formulacurvature3d} we obtain:
\begin{equation*}
\mR^h_{v,\xi} = \frac{\mu_2}{2}\,\eta\wedge v  \, , \qquad \mR^h_{v_1,v_2} = \frac{\mu_2 - 2\mu_1}{2}\,v_1\wedge v_2 \, .
\end{equation*}

\noindent
where $v ,v_1,v_2\in \mathfrak{X}(\Sigma)$ are orthogonal to $\xi$. Taking the interior product with $\xi$ in the first equation above we obtain:
\begin{equation*}
\frac{\mu_2}2 v=\mR^h_{v ,\xi}\xi=-\nabla_\xi^h \nabla_v^h \xi-\nabla_{[v,\xi]}^h \xi=-\nabla_\xi^h (\cC(v))+\cC(\nabla_\xi^h v)-\cC^2(v)=-(\cC^2+\nabla^h_\xi \cC)(v)\,. 
\end{equation*} 

\noindent
This shows that $\nabla_\xi^h \cC$ restricted to $\cH$ is a multiple of $\mathrm{Id}_{\cH}$, whence it vanishes since it is trace-free. We conclude:
\begin{equation}
\nabla_\xi^h \cC=0,\qquad \cC^2=-\frac{\mu_2}2\mathrm{Id}_{\cH}\, .
\end{equation}

\noindent
Clearly $(\mathcal{L}_{\xi}\cC)(\xi) = 0$. For $v\in \cH$, we compute:
\begin{equation*}
(\mathcal{L}_{\xi}\cC)(v) = \mathcal{L}_{\xi}(\cC(v)) -\cC(\mathcal{L}_{\xi}v)  =  \nabla^h_{\xi}(\cC(v)) - \nabla^h_{\cC(v)}\xi - \cC(\nabla^h_{\xi}v) + \cC(\nabla^h_v\xi) = 0\, ,
\end{equation*}

\noindent
upon use of $\nabla^h_{\xi}\cC = 0$. Furthermore, we have:
\begin{eqnarray*}
& (\mathcal{L}_{\xi}\cA)(v) = \mathcal{L}_{\xi}(\cA(v)) -\cA(\mathcal{L}_{\xi}v)  =  \nabla^h_{\xi}(\cA(v)) - \nabla^h_{\cA(v)}\xi - \cA(\nabla^h_{\xi}v) + \cA(\nabla^h_v\xi)\\
& = -\cC(\cA(v)) + \cA(\cC(v)) = - 2\cS\cA(v)\, ,
\end{eqnarray*}

\noindent
where we have used $\nabla^h_{\xi} \cA = 0$. A similar computation, using $\nabla^h_{\xi} \cS = 0$ shows that $\mathcal{L}_{\xi}\cA = 2\cS\cA$ whence $\mathcal{L}_{\xi}\cC = 0$. The last equation in the statement is a direct consequence of Cartan's formula for the Lie derivative of a form and hence we conclude.
\end{proof}

\noindent
In the following result, we denote by $t$ the Cartesian coordinate of $\mathbb{R}$ and we denote by $\mathbb{H}$ the three-dimensional hyperbolic space equipped with a metric of constant negative sectional curvature. Furthermore, we denote by $\mathrm{E}(1,1)$ the simply connected group of rigid motions of two-dimensional Minkowski space. This is a solvable and unimodular Lie group, see \cite{Milnor} for more details.

\begin{theorem}
\label{thm:solutions}
Let $M$ be a compact and oriented four-manifold and $\kappa > 0$. A non-flat pair $(g,\alpha)\in \Conf_{\kappa}(M)$ is a null Heterotic soliton with parallel torsion if and only if: 

\

\begin{enumerate}[leftmargin=*]
\item Relations $ \kappa \vert\alpha\vert^2_g = 1$  and $(\mu_1 = -\frac{1}{4\kappa} , \mu_2 =  \frac{1}{4\kappa})$ hold. In particular, there exists a double cover of $(\Sigma,h)$ that admits a Sasakian structure $(h_S,\xi_S)$ determined by:
\begin{equation*}
\xi_S := \sqrt{\dfrac{\mu_2}{2}} \xi\, , \quad \mathrm{Ric}^h(\xi) = \frac{1}{4\kappa}\xi\, , \quad \vert\xi\vert^2_h = 1\, , \quad \xi\in \mathfrak{X}(M)\, ,
\end{equation*}

\noindent
as well as:
\begin{equation*}
	h_S(v_1,v_2)= \left\lbrace \begin{matrix}
 - 2 h(\cA\circ \cC (v_1) , v_2) & \mathrm{if} \quad v_1,v_2 \in \cH \\ \\
   0 &  \mathrm{if} \quad v_1 \in \cH,\ v_2 \in \mathrm{Span}(\xi) \\ \\
 \dfrac{\mu_2}{2}h( v_1 , v_2) & \quad \,   \mathrm{if} \quad v_1,v_2 \in \mathrm{Span}(\xi)  
\end{matrix} \right.    
\end{equation*}

\noindent
where $(\Sigma,h)$ denotes the typical leaf of the foliation $\cF_{\alpha}\subset M$ defined by $\alpha$.

\

\item  Relation $ \kappa \vert\alpha\vert^2_g = 1$  holds and the lift $(\hat{g},\hat{\alpha})$ of $(g,\alpha)$ to the universal cover $\hat{M}$ of $M$ is isometric to either $\mathbb{R}\times \widetilde{\mathrm{Sl}}(2,\mathbb{R})$ or $\mathbb{R}\times \mathrm{E}(1,1)$ equipped with a left-invariant metric with constant principal Ricci curvatures given by $(0,0,-\frac{1}{2\kappa})$ and $\hat{\alpha} = \vert\alpha\vert_g \dd t$.
	
\
	
\item Relation $ \kappa \vert\alpha\vert^2_g = \frac{3}{2}$  holds and the lift $(\hat{g},\hat{\alpha})$ of $(g,\alpha)$ to the universal cover $\hat{M}$ of $M$ is isometric to $\mathbb{R}\times \mathbb{H}$ equipped with the standard product metric of scalar curvature $-\frac{3}{4\kappa}$ and $\hat{\alpha} = \vert\alpha\vert_g \dd t$.
\end{enumerate}
\end{theorem}

\begin{remark}
Reference \cite[Corollary 4.7]{Milnor} proves that both $\widetilde{\mathrm{Sl}}(2)$ and $\mathrm{E}(1,1)$ do admit Riemannian metrics with Ricci principal curvatures $(0,0,-\frac{1}{2\kappa})$.
\end{remark}

\begin{proof}
Let $(g,\alpha)\in\Sol_{\kappa}(M)$. To prove the statement it is enough to assume that $M$ is a simply connected four-manifold admitting a co-compact discrete group acting on $(M,g)$ by isometries that preserve $\alpha$. In that case, $(M,g) = (\mathbb{R}\times \Sigma, \dd t^2 + h)$ and $\alpha = \vert\alpha\vert_g \dd t$, see Remark \ref{remark:universalcover}. Assume first that $\mu_1 = \mu_2$. Then, Proposition \ref{prop:muconditions} immediately implies that $(\Sigma,h)$ is isometric to $\mathbb{H}$ equipped with the standard metric of scalar curvature $-\frac{3}{4\kappa}$, whence item $(3)$ follows. Therefore assume that $\mu_1 \neq \mu_2$ and, possibly going to a double cover, denote by $\xi$ a unit-norm eigenvector of $\mathrm{Ric}^h$ with simple eigenvalue $\mu_2$. Furthermore, assume $\mu_2 \neq 0$ since by Proposition \ref{prop:muconditions}, $\mu_2=0$ is not allowed. Using the notation introduced in Lemma \ref{lemma:equationsC}, consider the decomposition $\cC=\cS+\cA$ into its symmetric and skew-symmetric parts and let $\Sigma_0 \subset \Sigma$ denote a connected component of the open set of $\Sigma$ where $\cS$ and $\cA$ are both non-vanishing. Since $\cA$ is skew and $\tr(\cS)=0$, there exist smooth positive functions $\frs$ and $\fra$ on $\Sigma_0$ with $\cS^2=\frs^2\Id_{\cH}$ and $\cA^2=-\fra^2\Id_{\cH}$. By Lemma \ref{lemma:equationsC} we obtain: 
\begin{equation}
\label{sa}
\fra^2= \frs^2 + \frac{\mu_2}{2}\, .
\end{equation}

\noindent
On $\Sigma_0$ we can diagonalize $\cS$ through a smooth orthonormal frame $(u_1,u_2)$ satisfying $\cS(u_1)=\frs u_1$ and $\cS(u_2)=-\frs u_2$. Moreover, by replacing $u_2$ with its opposite if necessary, we can assume that $\cA(u_1)=\fra u_2$ and $\cA(u_2)=-\fra u_1$. We have:
\begin{equation}
\label{nxi}
\nabla_{u_1}^h \xi=\frs\, u_1 + \fra\, u_2\, ,  \qquad \nabla_{u_2}^h \xi = - \frs\, u_2 - \fra\, u_1\, .
\end{equation}

\noindent
Moreover, by Lemma \ref{lemma:equationsC} we have $\nabla_\xi^h \cS=0$ and $\nabla_\xi^h \cA =0$, which, together with the assumption $\mu_2 \neq 0$ implies:
\begin{equation}
\label{nxi1}
\xi(\fra)=\xi(\frs)=0, \qquad \nabla_\xi^h u_1 = \nabla_\xi^h u_2 = 0\, .
\end{equation}

\noindent
Furthermore, Equation \eqref{nxi} implies $h([u_1,u_2],\xi)=-2\fra$, so there exist two smooth functions $a ,b$ on $M_0$ such that $[u_1,u_2]=a u_1+b u_2-2\fra\xi.$ The Koszul formula then gives:
\begin{equation}
\label{nxy}
\nabla_{u_1}^h u_2=a\, u_1-\fra\,\xi\, , \,\, \nabla_{u_2}^h u_1 = -b\, u_2+\fra\,\xi\, ,\,\, \nabla_{u_1}^h u_1=-a\, u_2-\frs\,\xi\, ,\,\, \nabla_{u_2}^h u_2= b\, u_1+ \frs\, \xi\, .
\end{equation}

\noindent
Using Lemma \ref{lemma:equationsC} as well as equations \eqref{nxi}--\eqref{nxy} we can compute the following components of the Riemann tensor of $h$ along $\Sigma_0$, which must vanish as a consequence of Remark \ref{remark:formulacurvature3d}. We obtain:
\begin{eqnarray*}
0 &=& \mR^h_{u_1,\xi}u_2 = -\nabla_\xi^h(a\, u_1-\fra\,\xi)-\nabla_{\frs\,u_1 + \fra\, u_2}^h u_2\\
&=&-\xi(a) u_1-\frs(a\, u_1 - \fra\,\xi) - \fra\,(b\, u_1 + \frs\,\xi)=-(\xi(a)+\frs\,a + \fra\, b)\, u_1\, ,\\
0 &=& \mR^h_{u_2,\xi}u_1=-\nabla_\xi^h(-b\,u_2 + \fra\,\xi)+\nabla_{\fra\, u_1+ \frs \, u_2}^h u_1\\
&=&\xi(b) u_2 - \fra\,(a\, u_2 + \frs\,\xi)-\frs\,(-b\, u_2 + \fra\,\xi)=(\xi(b) - \fra\, a - \frs\,b) u_2\, ,\\
0 &=& \mR^h_{u_1,u_2}\xi=-\nabla_{u_1}^h(\fra\,u_1 + \frs\, u_2)-\nabla_{u_2}^h(\frs\, u_1+\fra\, u_2)-\nabla_{a \, u_1 + b\, u_2 - 2 \fra \, \xi}^h \xi\\
&=&-u_1(\fra) u_1 + \fra(a \,u_2 + \frs\,\xi)-u_1(\frs)u_2-\frs\,(a\,u_1-\fra\,\xi)-u_2(\frs)\, u_1 - \frs\, (-b\, u_2 + \fra\,\xi)\\
&&-u_2(\fra)u_2 - \fra\,(b\, u_1 + \frs\,\xi)-a\,(\frs\,u_1+\fra\, u_2)+b\,(\fra\, u_1+\frs\, u_2)\\
&=&-(u_1(\fra)+u_2(\frs)+2\frs\, a)\, u_1-(u_1(\frs) + u_2(\fra)-2\frs\, b)\, u_2\, .
\end{eqnarray*}

\noindent 
We thus have at each point of $\Sigma_0$:
\begin{equation}
\label{ab}
\xi(a)=-(\frs\, a + \fra\,b),\qquad \xi(b) = \fra\,a + \frs\,b\, ,
\end{equation}
\begin{equation}
\label{alphabeta}
u_1(\fra)+u_2(\frs)+2\frs\,a = 0\, , \qquad u_1(\frs) + u_2(\fra)-2\frs\,b = 0\, .
\end{equation}
Note that by \eqref{sa} we also have:
\begin{equation}
\label{alphabeta1}
\fra\, u_1(\fra) = \frs\, u_1(\frs)\, ,\qquad \fra\, u_2(\fra) = \frs\, u_2(\frs)\, .
\end{equation}

\noindent
We consider now the cases $\mu_2 <0$ and $\mu_2 >0$ separately.

\noindent{\bf Case 1}: $\mu_2 < 0$. From \eqref{sa} we have $\frs^2>0$ on $\Sigma$. In particular, $u_1$ and $u_2$ are smooth vector fields on $\Sigma$, and $a$ and $b$ are smooth functions on $\Sigma$. Applying $\xi$ to Equation \eqref{ab} and using Equation \eqref{nxi1} we get:
\begin{equation}
\label{xa}
\xi(\xi(a)) = - \frs\, \xi(a) - \fra\,\xi(b) = (\frs^2 - \fra^2)\, a = -\frac{\mu_2}{2} a\, ,
\end{equation}
and similarly $\xi(\xi(b))=-\frac{\mu_2}{2} b.$ The assumption that $(\mathbb{R}\times \Sigma, \dd t\otimes \dd t + h)$ has a co-compact discrete group $\Gamma$ acting freely by isometries implies that $a$ and $b$ are bounded functions on $\Sigma$. Indeed, each $\gamma\in\Gamma$ preserves the Ricci tensor of $(\mathbb{R}\times \Sigma, \dd t^2 + h)$, so $\gamma_*u_1=\pm u_1$ and $\gamma_* u_2 = \pm u_2$. Thus $a(x)=\pm a(\gamma(x))$ and $b(x)=\pm b(\gamma(x))$ for every $x\in \mathbb{R}\times \Sigma$ and $\gamma\in\Gamma$. By co-compactness of $\Gamma$, this shows that $a$ and $b$ are bounded.

Let $x\in \Sigma_0$ be some arbitrary point. Since $\xi$ is a geodesic vector field and the curve $c(t):=\exp_x(t\xi)$ satisfies $\dot{c}(t)=\xi_{c(t)}$ for every $t$, then $a$ is constant along $c(t)$ and in particular non-vanishing. Thus $c(t)\in \Sigma_0$ for all $t$. By \eqref{xa} the function $f:=a\circ c$ satisfies the ordinary differential equation $f''=-\frac{\mu_2}{2} f$. Thus $f$ is a linear combination of $\cosh(\sqrt{-\frac\mu2}t)$ and $\sinh(\sqrt{-\frac\mu2}t)$. Therefore, since $f$ is bounded, it has to vanish. In particular $a(x)=0$, and since $x$ was arbitrary, $a=0$ on $\Sigma_0$. Similarly, $b=0$ on $\Sigma_0$. By \eqref{alphabeta} and \eqref{alphabeta1}, we obtain:
\begin{equation*}
\frs^2\, u_1(\frs) = \frs\,\fra u_1(\fra) = -\frs\,\fra\,u_2(\frs) = - \fra^2 u_2(\fra) = \fra^2 u_1(\frs)\, ,
\end{equation*}

\noindent
whence $u_1(\frs)=0$. Similarly we obtain $u_2(\frs)=0$, thus showing that $\fra$ and $\frs$ are constant on $\Sigma_0$. In particular, $\Sigma_0$ is open and closed in $\Sigma$, so either $\Sigma_0=\Sigma$ and $\fra$ is non-vanishing, or $\Sigma_0$ is empty and $\fra=0$ on $\Sigma$. If $\Sigma_0$ was empty, then $\mathcal{C}=0$ and $\mu_2=0$, which is not possible. Hence $\Sigma_0=\Sigma$ and all equations above are valid on $\Sigma$. The orthonormal frame $(\xi,u_1,u_2)$ satisfies:
\begin{equation*}
[\xi,u_1]=-(\frs u_1+ \fra u_2),\qquad [\xi,u_2]= \frs u_2 + \fra u_1,\qquad [u_1,u_2]=-2\fra\,\xi\, ,
\end{equation*}

\noindent
hence $\Sigma$ is an unimodular Lie group equipped with a left-invariant metric $h$. The Killing form of its Lie algebra $\mathfrak{g}$ can be easily computed to be:
\begin{equation*}
B(\xi,\xi)=-\mu_2\, ,\,\, B(u_1,u_1)=B(u_2,u_2)=-4 \fra^2\, ,\,\, B(u_1,u_2)= 4\fra\frs\, ,\,\, B(u_1,\xi)=B(u_2,\xi)=0\, .
\end{equation*}

\noindent
For $\fra\ne 0$, $B$ is non-degenerate and has signature $(2,1)$, so $\mathfrak{g}$ is isomorphic to $\mathfrak{sl}(2,\mathbb{R})$. If $\fra = 0$, $\mathfrak{g}$ is solvable and isomorphic to a semi-direct product $\mathbb{R}\ltimes\mathbb{R}^2$ that can be identified with the Lie algebra of E$(1,1)$, the group of rigid motions of Minkowski two-dimensional space. In both cases we can easily compute using \eqref{nxi} and \eqref{nxy}:
\begin{eqnarray*}
& \mR^h_{u_1,u_2}u_1=\nabla_{u_1}^h (\fra\,\xi)-\nabla_{u_2}^h(-\frs\,\xi)-\nabla_{[u_1,u_2]}^h u_1 \\
& = \fra (\frs\, u_1 + \fra\, u_2) - \frs\,(\fra\,u_1 + \frs u_2)= (\fra^2-\frs^2) u_2=\frac{\mu_2}{2} u_2\, ,
\end{eqnarray*}

\noindent
which by Lemma \ref{lemma:equationsC} implies $\mu_1=0$, in agreement with Proposition \ref{prop:muconditions}. This proves item $(2)$.

\noindent{\bf Case 2}: $\mu_2 > 0$.   We define the following endomorphism $\Psi \in\End(T\Sigma)$ of $T\Sigma$:
\begin{equation*}
 \Psi(\xi) = 0\, , \qquad \Psi(v) = -\sqrt{\frac{2}{\mu_2}}\cC(v)\, , \qquad \forall \,\, v\in \cH\, .
\end{equation*}

\noindent
Define $\xi_S=\sqrt{\dfrac{2}{\mu_2}} \xi$ and $\eta_S=\sqrt{\dfrac{\mu_2}{2}} \eta$. Clearly:
\begin{equation*}
\Psi(\xi_S) = 0\, , \qquad \eta_S(\xi_S)=1 \, , \qquad \Psi^2 = - \Id^2 + \xi_S\otimes\eta_S\, .
\end{equation*}

\noindent
Moreover, define the symmetric tensor $h_S \in \mathrm{Sym}^2(T^* \Sigma)$ as follows:
\begin{equation*}
h_S(v_1,v_2)= \left\lbrace \begin{matrix}
 - 2 h(\cA\circ \cC (v_1) , v_2) & \mathrm{if} \quad v_1,v_2 \in \cH \\ \\ \dfrac
{\mu_2}{2}h( v_1 , v_2) & \quad \, \,  \mathrm{if} \quad v_1,v_2 \in \mathrm{Span}(\xi_S) 
\end{matrix} \right. 
\end{equation*}

\noindent
we check that:
\begin{equation*}
h_S(\Psi (v_1) , \Psi (v_2)) = h_S(v_1 , v_2) - \eta_S(v_1)\, \eta_S(v_2)\, , \quad \forall\,\, v_1, v_2 \in T\Sigma\, .
\end{equation*}

\noindent
On the other hand:
\begin{equation*}
h_S(\Psi (v_1) ,  v_2) =-2 \sqrt{\frac{\mu_2}{2}}h(\cA (v_1) ,  v_2)  = -\dd \eta_S (v_1 , v_2) \, , \quad \forall\,\, v_1, v_2 \in T\Sigma\, .
\end{equation*}

\noindent
Furthermore, it can be verified that $h_S$ is non-degenerate since $\det(\cA\cC) > 0$, which in turn implies that $h_S$ is  positive definite. In addition, by Equation \eqref{sa}, we observe that $\fra$ is nowhere vanishing, implying that $\cA$ is nowhere singular. We infer that $\dd\eta_S \neq 0$ everywhere on $\Sigma$ and therefore $(\xi_S,\eta_S,\Psi)$ defines a contact structure on $\Sigma$ compatible with the Riemannian metric $h_S$. By Lemma \ref{lemma:equationsC} the Lie derivative $\mathcal{L}_{\xi_S} \Psi = 0$ vanishes whence $(h_S,\xi_S,\Psi)$ is K-contact structure on $\Sigma$, a condition that in three dimensions is well-known to be equivalent to $(h_S,\xi_S,\eta_S,\Psi)$ being Sasakian and hence we conclude. 
\end{proof}
 
\begin{remark}
In the cases in which the leaves of $\cF_{\alpha}\subset M$ are Sasakian three-manifolds, with respect to an \emph{auxiliary} metric as described in the previous theorem, their cone is a K\"ahler four-manifold, and in particular of special holonomy, whence realizing the proposal made in \cite{Grana:2014rpa,Grana:2014vxa} to \emph{geometrize} supergravity fluxes. While the occurrence of Sasakian structures in supersymmetric supergravity solutions is well-documented, see for instance \cite{Sparks:2010sn} and references therein, the natural appearance on these structures in a non-supersymmetric framework, such as the one considered here, was highlighted only recently in \cite{Murcia:2019cck}.
\end{remark}

\noindent
Theorem \ref{thm:solutions} can be used to construct large families of solutions of the Heterotic {\color{red} soliton} system. These are, to the best knowledge of the authors, the first solutions in the literature that are not locally isomorphic to a supersymmetric Heterotic solution. For example, as a direct consequence of Theorem \ref{thm:solutions} we have the obtain the following corollaries.

\begin{corollary}
\label{cor:examples1}
Every mapping torus of a complete hyperbolic three-manifold or a manifold covered by $\widetilde{\mathrm{Sl}}(2,\mathbb{R})$ or $\mathrm{E}(1,1)$ admits a null Heterotic soliton with parallel torsion.
\end{corollary}

\begin{corollary}
\label{cor:examplesSasakian}
Let $(h_S,\xi_S)$ be a Sasakian structure on $\Sigma$ with contact $1$-form $\eta_S$ satisfying:
\begin{equation*}
\mathrm{Ric}^{h_S} = -\frac{1}{2} h_S + \eta_S\otimes \eta_S \, .
\end{equation*}

\noindent
Then, the mapping torus of $(\Sigma, c^2 h_S)$ admits a null Heterotic soliton with parallel torsion for $c^2=2 \kappa$.
\end{corollary}

\begin{remark}
The Sasakian three-manifolds occurring in the previous corollary are a particular type of $\eta$-Einstein Sasakian manifolds, a class of Sasakian manifolds extensively studied in the literature, see for example \cite{Boyer:2004eh} and its references and citations.
\end{remark}

\noindent
The topology of the Heterotic solitons constructed in the previous theorem depends rather explicitly in the string slope parameter $\kappa$. Set $\vert\alpha\vert^2_g = 1/2$ for simplicity, whence $\kappa \in \left\{ 1, 2, 3\right\}$ is \emph{discrete}, and different values of $\kappa$ will correspond in general with Heterotic solitons of different topology. For example, if $\kappa =1$, $(M,g,\alpha)$ can be the suspension of a Sasakian three-manifold, if $\kappa = 2$ then $(M,g,\alpha)$ can become the suspension of a three-manifold covered by E$(1,1)$ or $\widetilde{\mathrm{Sl}}(2,\mathbb{R})$, and if $\kappa = 3$ then $(M,g,\alpha)$ can become the suspension of a hyperbolic three-manifold, which again results in a new topology change. We remark that for the Heterotic solitons described in Theorem \ref{thm:solutions} the limit $\kappa\to 0$ is not well-defined, whence they can be considered as genuinely \emph{stringy}.


\appendix


\phantomsection
\bibliographystyle{JHEP}

\begin{thebibliography}{100}
	



\bibitem{Ashmore:2019rkx}{A.~Ashmore, C.~Strickland-Constable, D.~Tennyson and D.~Waldram, {\sl Heterotic backgrounds via generalised geometry: moment maps and moduli}, JHEP \textbf{11} (2020), 071.}


\bibitem{Baraglia:2013wua}{D.~Baraglia and P.~Hekmati, {\sl Transitive Courant Algebroids, String Structures and T-duality}, Adv.\ Theor.\ Math.\ Phys.\  {\bf 19} (2015), 613.}

\bibitem{BRI}{E. Bergshoeff and M. de Roo, {\sl Supersymmetric Chern-Simons Terms in Ten-dimensions}, Phys. Lett. B {\bf 218} (1989), 210.}

\bibitem{BRII}{E. Bergshoeff and M. de Roo, {\sl The Quartic Effective Action of the Heterotic String and Supersymmetry}, Nucl. Phys. B {\bf 328} (1989), 439.}

\bibitem{Bergshoeff:1995cg}{E.~Bergshoeff, B.~Janssen and T.~Ortin, {\sl Solution generating transformations and the string effective action}, Class. Quant. Grav. \textbf{13} (1996), 321--343.}

\bibitem{Besse}{A. L. Besse, {\it Einstein Manifolds}, Classics in Mathematics, Springer (1987).}


\bibitem{Boyer:2004eh}{C.~P.~Boyer, K.~Galicki and P.~Matzeu, {\sl On eta-Einstein Sasakian geometry}, Commun. Math. Phys. \textbf{262} (2006), 177--208.}
%
%
%
\bibitem{CG1}{J. Cheeger and D. Gromoll, {\sl The splitting theorem for manifolds of non-negative Ricci curvature}, J. Differential Geom. {\bf 6} (1971) 119--128.}

\bibitem{CG2}{J. Cheeger and D. Gromoll, {\sl On the Structure of Complete Manifolds of Nonnegative Curvature}, Ann. of Math. {\bf 96} (3) (1972), 413--443.}

%
\bibitem{Coimbra:2014qaa}{A.~Coimbra, R.~Minasian, H.~Triendl and D.~Waldram, {\sl Generalised geometry for string corrections}, JHEP {\bf 11} (2014), 160.}

\bibitem{delBarco:2017qol}{V.~del Barco, L.~Grama and L.~Soriani, {\sl $T$-duality on nilmanifolds}, JHEP \textbf{05} (2018), 153.}
 
\bibitem{delaOssa:2014msa}{X.~de la Ossa and E.~E.~Svanes, \emph{Connections, Field Redefinitions and Heterotic Supergravity}, JHEP \textbf{12} (2014), 008.}

\bibitem{Ebin}{D. G. Ebin, {\sl The manifold of Riemannian metrics}, In Global Analysis
	(Proc. Sympos. Pure Math., Vol. XV, Berkeley, Calif., 1968), 11--40.
	Amer. Math. Soc., Providence, R.I., 1970.}

\bibitem{Fei}{ T. Fei, {\sl Generalized Calabi-Gray Geometry and Heterotic Superstrings}, \arxiv{1807.08737}.}

\bibitem{Fei:2018mzf}{T.~Fei, B.~Guo and D.~H.~Phong, \emph{Parabolic Dimensional Reductions of 11D Supergravity},
Commun. Math. Phys. {\bf 369} (2019), 811--836.}

\bibitem{Fei:2020kkl}{T.~Fei, D.~H.~Phong, S.~Picard and X.~Zhang, {\sl Geometric Flows for the Type IIA String},
 \arxiv{2011.03662}.}

\bibitem{Fei:2020nwv}{T.~Fei, D.~H.~Phong, S.~Picard and X.~Zhang, {\sl Estimates for a geometric flow for the Type IIB string},  \arxiv{2004.14529}.}

\bibitem{FigueroaGadhia}{J. M. Figueroa-O'Farrill and S. Gadhia, {\sl Supersymmetry and spin structures}, Class. Quant. Grav. {\bf 22} (2005) L121.}

%
\bibitem{YauV}{ J.-X. Fu and S.-T. Yau, {\sl A Monge-Amp\`ere type equation motivated by string theory}, Comm. Anal. Geom. {\bf 15} (2007) 29--76.}
 
\bibitem{YauVI}{ J.-X. Fu and S.-T. Yau, {\sl The theory of superstring with flux on non-K\"ahler manifolds and the complex Monge-Amp\`ere equation}, J. Differential Geom. {\bf 78} (2008) 369--428.}
 
\bibitem{Garcia-Fernandez:2013gja}{M.~Garcia-Fernandez, {\sl Torsion-free generalized connections and Heterotic Supergravity}, Commun. Math. Phys. \textbf{332} (1) (2014), 89--115.}
 
\bibitem{GFR}{M.~Garcia-Fernandez, {\sl Lectures on the Strominger system}, Travaux math\'ematiques, {\bf XXIV} (2016) 7--61.}

\bibitem{Garcia-Fernandez:2016ofz}{M.~Garcia-Fernandez, {\sl Ricci flow, Killing spinors, and T-duality in generalized geometry}, Adv. in Math. \textbf{350} (2019), 1059--1108.}
 
\bibitem{GFRST}{M. Garcia-Fernandez, R. Rubio, C. S. Shahbazi, C. Tipler, {\sl Canonical metrics on holomorphic Courant algebroids}, \arxiv{1803.01873}.}

\bibitem{GFRSTII}{M. Garcia-Fernandez, R. Rubio, C. S. Shahbazi, C. Tipler, {\sl Heterotic supergravity and moduli stabilization}, to appear.}

\bibitem{GFRTGauge}{M. Garcia-Fernandez, R. Rubio, C. Tipler, {\sl Gauge theory for string algebroids}, \arxiv{2004.11399}.}

\bibitem{FernandezStreets}{M. Garc\'ia - Fern\'andez and J. Streets,  {\it Generalized Ricci Flow}, MS University Lecture Series, 2020.}

\bibitem{Gauduchon}{P. Gauduchon, {\sl Structures de Weyl-Einstein, espace de twisteurs et vari\'et\'es de type $S^1 \times S^3$}, J. reine angew. Math. {\bf 469} (1) (1995), 1--50. }

%
\bibitem{Gran:2018ijr}{U.~Gran, J.~Gutowski and G.~Papadopoulos, {\sl Classification, geometry and applications of supersymmetric backgrounds}, Phys.\ Rept.\  {\bf 794} (2019) 1.}

\bibitem{Grana:2005jc}{M.~Gra\~na, {\sl Flux compactifications in string theory: A Comprehensive review}, Phys. Rept. \textbf{423} (2006), 91--158.}

\bibitem{Grana:2014rpa}{M.~Gra\~na, C.~S.~Shahbazi and M.~Zambon, {\sl Spin(7)-manifolds in compactifications to four dimensions}, JHEP \textbf{11} (2014), 046.}

\bibitem{Grana:2014vxa}{M.~Gra\~na and C.~S.~Shahbazi, {\sl M-theory moduli spaces and torsion-free structures}, JHEP \textbf{05} (2015), 085.}

%
%
\bibitem{Hall}{B.C. Hall, {\it Lie Groups, Lie Algebras, and Representations. An elementary Introduction}, Springer, Second Edition, 2015}
 
\bibitem{Hull}{C.M. Hull, {\sl Anomalies, ambiguities and superstrings}, Phys. Lett. B {\bf 167} (1986), 51--55.}

 \bibitem{HullTownsend}{C. Hull and P. Townsend, {\sl World-sheet supersymmetry and anomaly cancellation in
		the heterotic string}, Physics Letters B 178 (1986), no. 23 187--192.}
	
\bibitem{SIvanov}{S. Ivanov, {\sl Heterotic supersymmetry, anomaly cancellation and equations of motion}, Phys. Lett. B {\bf 685} (2010) 190--196.}

%
\bibitem{Koiso}{N. Koiso, {\sl Einstein metrics and complex structures}, Invent Math {\bf 73} (1983), 71--106.}

\bibitem{KoisoII}{N. Koiso, {\sl Yang-Mills connections and moduli space}, Osaka J. Math. {\bf 24} (1987), 147--171.}

\bibitem{KoisoIII}{N. Koiso, {\sl Rigidity and infinitesimal deformability of Einstein metrics}, Osaka J. Math. {\bf 19} (3) (1982), 643--668.}

\bibitem{McCullough}{D. McCullough, {\sl Isometries of Elliptic 3-Manifolds}, J. London Math. Soc. {\bf 65} (1) (2002), 167--182.}

\bibitem{Melnikov:2014ywa}{I.~V.~Melnikov, R.~Minasian and S.~Sethi, \emph{Heterotic fluxes and supersymmetry}, JHEP \textbf{06} (2014), 174.}

\bibitem{Milnor}{J. Milnor, {\sl Curvatures of Left Invariant Metrics on Lie Groups}, Adv. in Math. {\bf 21} (1976), 293--329.}

\bibitem{HeteroticRicciFlow}{A. Moroianu and C.~S.~Shahbazi, {\sl The Heterotic-Ricci flow and its three-dimensional solitons}, to appear.}

\bibitem{Murcia:2019cck}{\'A.~Murcia and C.~S.~Shahbazi, {\sl Contact metric three manifolds and Lorentzian geometry with torsion in six-dimensional supergravity}, J. Geom. Phys. \textbf{158} (2020), 103868.}

\bibitem{Oliynyk:2005ak}{T.~Oliynyk, V.~Suneeta and E.~Woolgar, \emph{A Gradient flow for worldsheet nonlinear sigma models},
Nucl. Phys. B \textbf{739} (2006), 441--458.}

\bibitem{Ortin}{T.~Ort\'in, {\it Gravity and Strings}, Cambridge Monographs on Mathematical Physics, 2nd edition, 2015.}

\bibitem{Pedersen}{H. Pedersen, Y. S. Poon and A. Swann, {\sl Einstein-Weyl deformations and submanifolds}, Int. J. Math. {\bf 7} (1996), 705--719.}

\bibitem{Phong}{Duong H. Phong, {\sl Geometric Partial Differential Equations from Unified String Theories}, Contribution to the Proceedings of the ICCM 2018, Taipei, Taiwan.}

\bibitem{Phong:2018wgn}{D.~H.~Phong, S.~Picard and X.~Zhang, \emph{New curvature flows in complex geometry}, Surveys in Differential Geometry Volume 22 (2017).}

\bibitem{PhongPicardZhang}{D.H. Phong, S. Picard, and X.W. Zhang, \emph{Geometric flows and Strominger systems}, Math. Z. {\bf 288} (2018) 101--113.}

\bibitem{Polchinski:1998rq}{J.~Polchinski, {\it String theory. Vol. 1: An introduction to the bosonic string}, Cambridge Monographs on Mathematical Physics (1998).}

\bibitem{Polchinski:1998rr}{J.~Polchinski, {\it String theory. Vol. 2: Superstring theory and beyond}, Cambridge Monographs on Mathematical Physics (1998).}


\bibitem{Sen}{A. Sen, {\sl (2, 0) supersymmetry and space-time supersymmetry in the heterotic string
		theory}, Nuclear Physics B 278 (1986), no. 2 289--308.}

\bibitem{Sparks:2010sn}{J.~Sparks, {\sl Sasaki-Einstein Manifolds}, Surveys Diff. Geom. \textbf{16} (2011), 265--324.}

\bibitem{Streets1}{J. Streets, {\sl Regularity and expanding entropy for connection Ricci flow}, J. Geom. Phys. {\bf 58} (7) (2008), 900--912.}

\bibitem{Streets2}{J. Streets, {\sl Generalized geometry, T-duality, and renormalization group flow},  J. Geom. Phys. {\bf  114} (2017), 506--522.}
 
\bibitem{StreetsA}{J. Streets, {\sl Generalized K\"ahler–Ricci flow and the classification of nondegenerate generalized Kähler surfaces}, Adv. in Math. {\bf 316} (20) (2017), 187--215.}

\bibitem{StreetsYuryII}{J. Streets and Y. Ustinovskiy, {\sl Classification of generalized K\"ahler-Ricci solitons on complex surfaces}, \arxiv{1907.03819}.}

\bibitem{StreetsYury}{J. Streets and Y. Ustinovskiy, {\sl The Gibbons-Hawking ansatz in generalized K\"ahler geometry}, \arxiv{2009.00778}.}

\bibitem{Strominger}{ A. Strominger, {\sl Superstrings with torsion}, Nuclear Physics B {\bf 274} (2) (1986), 253--284.}

%
%

\bibitem{YauIII}{S.-T. Yau, {\sl Complex geometry: Its brief history and its future}, Science in China Series A Mathematics {\bf 48} (2005), 47--60.}

\bibitem{YauIV}{S.-T. Yau, {\sl Metrics on complex manifolds}, Science in China Series A Mathematics {\bf 53} (2010), 565--572.}
\end{thebibliography}


\end{document}